\documentclass[a4paper,10pt]{article} %bollanonlin23jul22.tex regi verzio

\usepackage{xcolor}
\usepackage{graphicx}
\usepackage{amsfonts}
\usepackage{amstext}
\usepackage{subfigure}
\usepackage{amsthm}
\usepackage{amssymb,amsmath,amsbsy}
\usepackage[T1]{fontenc}

\usepackage[utf8]{inputenc}
\usepackage{authblk}
\usepackage{caption}
\usepackage{bm}
\usepackage{comment}

\newcommand{\PPP}{\mathbb{P}}
\newcommand{\R}{\mathbb{R}}
\newcommand{\E}{\mathbb{E}}
\newcommand{\0}{\mathbf{0}}
\newcommand{\1}{\mathbf{1}}

\newcommand{\rrr}{\mathbf{r}}
\newcommand{\x}{\mathbf{x}}

\newcommand{\uuu}{\mathbf{u}}
\newcommand{\vvv}{\mathbf{v}}

\newcommand{\I}{\bm{I}}

\newcommand{\M}{\bm{M}}
\newcommand{\N}{\bm{N}}
\newcommand{\K}{\bm{K}}

\newcommand{\V}{\bm{V}}
\newcommand{\W}{\bm{W}}
\newcommand{\A}{\bm{A}}
\newcommand{\B}{\bm{B}}
\newcommand{\C}{\bm{C}}
\newcommand{\G}{\bm{G}}
\newcommand{\DD}{\bm{D}}
\newcommand{\PP}{\bm{P}}

\newcommand{\RR}{\bm{R}}

\newcommand{\T}{\bm{T}}
\newcommand{\ep}{\varepsilon}
\newcommand{\epv}{\bm{\ep}}

\newcommand{\Dada}{\bm{\Delta}}

\newcommand{\diag}{\mathrm{diag}}

\newcommand{\rk}{\mathrm{rank}\,}
\newcommand{\Var}{{\mathrm{Var}}}

\newtheorem{theorem}{Theorem}
\newtheorem{proposition}{Proposition}

\newtheorem{remark}{Remark}
\newtheorem{corollary}{Corollary}

\title{Structure and Noise in Dense and Sparse Random Graphs:
   Percolated Stochastic Block Model via the EM Algorithm and
 Belief Propagation with Non-Backtracking Spectra}

\author{Marianna Bolla \thanks{Department of Stochastics, Budapest University
 of Technology and Economics, M{\H u}\-egyetem rkp. 3, Budapest 1111, Hungary.
 E-mail: marib@math.bme.hu} 
\and Hannu Reittu \thanks{VTT Technical Research Center of Finland, Tietotie 3. FI-2150 Espoo, Finland, E-mail: hannu.reittu@vtt.fi} 
\and  Runtian Zhou \thanks{Budapest Semesters in Mathematics; also at Duke
University, 2138 Campus Dr, Durham, NC 27708, USA. E-mail: rz169@duke.edu}
}

\begin{document}

\maketitle

\section*{Abstract}

In this survey paper it is illustrated
how spectral clustering methods for unweighted graphs
are adapted to the dense and sparse regimes. 
Whereas Laplacian and modularity based spectral clustering is apt to dense
graphs, recent results show that for sparse ones,
the non-backtracking spectrum is
the best candidate to find assortative clusters of nodes.
Here belief propagation in the sparse stochastic block model is derived 
with arbitrarily given model parameters that results in a non-linear system of
equations; with linear approximation, the spectrum of the non-backtracking
matrix is able to specify the number $k$ of clusters. Then the 
model parameters themselves can be estimated by the EM algorithm.

Bond percolation in the assortative model is  considered in the following
two senses: the within- and between-cluster edge probabilities
decrease with the number of nodes and  edges coming into existence in this 
way are retained with probability $\beta$. As a consequence, the optimal
$k$ is the number of the structural real eigenvalues
(greater than $\sqrt{c}$, where $c$ is the average degree) 
of the non-backtracking matrix of the graph.
Assuming, these eigenvalues $\mu_1 >\dots > \mu_k$ are distinct, the
multiple phase
transitions obtained  for $\beta$ are $\beta_i =\frac{c}{\mu_i^2}$; further, at
$\beta_i$ the number of detectable clusters is $i$, for $i=1,\dots ,k$.
Inflation--deflation techniques are also discussed to classify the nodes
themselves, which can be the base of the sparse spectral clustering.
Simulation results, as well as real life examples are presented.

\vskip0.2cm  
\noindent
\textbf{Keywords}: stochastic block model, belief propagation,
EM algorithm, non-backtracking  spectra, bond percolation.

\vskip0.2cm
\noindent
\textbf{Mathematics Subject Classification}: 05C50, 05C80, 62H30 

\section{Introduction}\label{intro}

In~\cite{Bolla20} we considered generalized quasirandom properties of
expanding quasirandom graph sequences, which are deterministic
counterparts of generalized random graphs, where the probabilities of edges
coming into existence only depend on the cluster memberships of the endpoints,
and are fixed during the expansion. 
If the number of the underlying
clusters (number of nodes of the so-called model graph of~\cite{LovSos})
is $k$, then a dense graph sequence is obtained in this way (the average
degree of nodes is proportional to $n$, %but it is greater than $\log n$,
where $n$ is the number of
nodes at a stage of the expansion). We proved that
the adjacency matrix of such a random or quasirandom graph $G_n$ coming from
the $k$-cluster model
has $k$ structural eigenvalues (of order $n$), while
the others are $o(n)$ in absolute value. Also, the normalized modularity
matrix defined in~\cite{Bolla11} (its spectrum is in [-1,1]) has $k$ eigenvalues
separated from 0, whereas the others are $o(1)$. Further, the clusters
can be recovered by the weighted k-means algorithm applied to the
$(k-1)$-dimensional representatives of the nodes obtained via the
eigenvectors corresponding to the $k-1$ non-trivial structural eigenvalues
(more precisely, the objective function of the k-means clustering is
$o(1)$).
In case of expanding dense regular graphs,
the coordinates of any adjacency eigenvector
(except the trivial one) are concentrated around certain values
(according to a Gaussian distribution), keeping the orthogonality,
see~\cite{Backhausz}. 
However, in case of sparse graphs, the eigenvectors of
the usual graph-based matrices are concentrated on the high degree nodes,
see, e.g.,~\cite{Benaych}, and so, they are not able to cluster the nodes.

Consequently, the adjacency, Laplacian, or modularity matrix based
methodology is not applicable in the sparse case, see~\cite{Stephan}.
To consider this situation, the edge probabilities are rescaled,
with $n$ or $\log n$. These possibilities are discussed in~\cite{Abbe}.
In~\cite{Mossel,MosselComb}, the authors consider scaling with $n$, and
refine the block model threshold conjecture; they specify  when
information theoretically is possible or impossible to distinguish between the
clusters, depending on the model parameters. This so-called Kesten--Stigum
threshold will be discussed in Section~\ref{pre}.
Bickel and Chen~\cite{Bickel} give a nonparametric statistical view of
network models, and consider consistency in the strong or weak sense,
when the estimated cluster assignment approaches the true one as
$n\to\infty$. They prove that the likelihood modularity is always
consistent, while the Newman--Girvan modularity is consistent only
under stronger assumptions.  They also consider different scalings.
Other papers, instead of consistency, define different measures
for the agreement of the true and the recovered clusterings (correlation,
proportions of misclassifications, and information theoretical measures,
e.g., the Kullbach--Leibler divergence).
In~\cite{Benaych}, the transition happens at $o (\log n )$ average degrees.
At $\Omega (\log n)$ (e.g., poly-$\log n$) average degrees the
paper~\cite{Coste} establishes results %, and there are many possibilities
in this intermediate regime.

In the Laplacian or modularity based
spectral clustering~\cite{Bolla13}, $k$ is specified via spectral gaps,
and the distance between the eigen-subspace spanned by the eigenvectors
corresponding to the structural eigenvalues and that of the step-vectors is
estimated by the Davis--Kahan theorem (see, e.g.,~\cite{Stewart}).
As the square of this distance is
the objective function of the k-means algorithm, this is at our
disposal to find  the clusters approximately in the dense case.
This concept can be extended to non-symmetric matrices, like to the
non-backtracking one, and due to theorems of~\cite{Stephan} that use
Bauer--Fike type perturbation results, similar
subspace perturbations are applicable to sparse graphs too.

In the sparse case we scale with $n$, so the average degree of the graph is of
constant order. The so obtained sparse stochastic block model (called
$SBM_k$)
%behind an expanding
%graph sequence (the number $n$ of nodes is increasing) with given number 
%($k$) of clusters which are
is assortative in
the sense that the within-cluster probability of two nodes being connected
is higher than the between-cluster one. 
Here bond percolation has two meanings:
the first is that the within- and between-cluster edge probabilities
decrease with $n$, and the second is that edges coming into existence in this 
way are retained with probability $\beta$. A real-life 
social network is an instance of such an expanding sequence, where
two persons are connected randomly, with 
probability depending on their cluster memberships and this probability is
the ``smaller'' the ``larger'' the network is; 
further, along these connections an information (e.g., an infection)
is transmitted only with probability $\beta$ (which is related to the 
seriousness of the epidemic, akin to the recombination ratio).

The question is that for which values of the model parameters and $\beta$
the  $k$ clusters can be distinguished with high probability (w.h.p.). 
If we knew $k$, the model parameters could be estimated with the 
EM algorithm, see~\cite{Bolla13}. 
The method of belief propagation (BP) is able to
relate the model parameters to the non-backtracking spectrum of the graph,
see~\cite{Decelle,Karrer,Krzakala,Torresetal}.
However, the detectability thresholds 
have been given only in special cases (e.g., in the symmetric case
of~\cite{Decelle,Moore}, to be described in Section~\ref{pre}). 
In the present paper, BP is derived in the generic case; 
it provides a non-linear system of equations for the 
conditional probabilities of the
nodes belonging to the clusters (given their neighbors), where the number of
equations is $2mk$ ($m$ is the
number of edges in the underlying graph). Though, there are numerical
algorithms to solve this system of equations, 
BP is mainly important from
theoretical point of view: via linearization and stability issues, 
our Theorem~\ref{bpgeneral} establishes a close
relation between the model parameters and the structural eigenvalues of the
non-backtracking matrix $\B$ of the
graph under mild assumptions. In this way, the number $k$ of clusters is
approximately equal to the number of the structural eigenvalues of $\B$;
but for given $k$, it is the EM algorithm that has the massive theoretical 
background to estimate the model parameters themselves.

We also experienced that, surprisingly, the leading $k$ (structural) eigenvalues
of the expected adjacency matrix
of a random graph coming from the sparse $SBM_k$
model are w.h.p. closer to the structural eigenvalues of $\B$
than to those of the adjacency matrix $\A$ of a randomly generated graph
from this model,  under some
balancing conditions for the cluster sizes and average degrees of the clusters. 
This seems to contradict to the laws of large numbers, but in the sparse 
case (possessing nearly constant average degrees) 
it is supported by %computations, simulations, and  also by
theoretical considerations as for deformed Wigner matrices
(see, e.g.,~\cite{Capitaine,Stephan}).
Roughly speaking, here the amount of the noise in the data suppresses the
structure, and it is impossible to find the clusters with the traditional
spectral clustering methods. Luckily, in this case,
the non-backtracking matrix is at our disposal.

Since the structural eigenvalues of $\B$ and those of
the expected adjacency matrix
are closely aligned, the so-called detectability threshold and
the phase transitions for $\beta$ also depend on these eigenvalues.
However, in practical situations, only the non-backtracking spectrum is at our
disposal and we do not know the model parameters. 
Based on this, we conclude the value of $k$ and
derive multiple phase transitions  in the general  $SBM_k$ model, where,
in addition, the edges of the graph coming from this model
are retained with probability $\beta$, 
and we have the further percolated sparse stochastic block model 
$SBM_k^{\beta}$ too. 

The rest of the paper is organized as follows.
In Section~\ref{pre},  properties of the non-backtracking matrix are
discussed together with basic notions and notation; further, the
sparse stochastic block model and bond-percolation models are introduced.
In Section~\ref{bp}, we perform BP theoretically which shows that $k$ is about
the number of the structural eigenvalues of the non-backtracking matrix $\B$
that are larger than $\sqrt{c}$, where $c$ is the average degree of the
random graph $G_n \in SBM_{k}$;
it is Theorem~\ref{bpgeneral} that supports this choice of $k$.
In Section~\ref{em}, we run the EM algorithm to estimate the parameters
 of the $SBM_k$ model that fits to our observed graph $G_n$.
In the possession of the new estimates for the clusters and their average
degrees, we adjust the number $k$ so that possible eigenvalues of $\B$
around $\sqrt{c}$ are added or deleted, and $k_0$ denotes the new number of
clusters. In Section~\ref{id}, we use inflation--deflation techniques for this
purpose and apply a theorem of~\cite{Stephan} to classify the majority of
nodes by the k-means algorithm; this issue is summarized in Theorem~\ref{kvar},
further, by a corollary and remark after it.
Eventually, we find number $k$ of phase transitions for $\beta$
in the $SBM_k^{\beta}$ model, when
$1,2,\dots ,k$ clusters can be detected and supported by
simulations in Section~\ref{betaperc}.
In Section~\ref{appl}, real life application for quantum chemistry graphs is
presented.
Section~\ref{conc} is devoted to conclusions and further perspectives.

\section{Preliminaries}\label{pre}

\subsection{Non-backtracking matrix of unweighted graphs}

The \textit{Hashimoto edge-incidence matrix}, with other words, 
\textit{non-backtracking matrix}
$\B$ of a simple graph $G$ on $n$ nodes and $m$ edges  is defined as
a $2m \times 2m$ 
non-symmetric matrix of 0-1 entries, see~\cite{Karrer,Krzakala,Martin}
in context of non-backtracking random walks, community detection, and 
centrality. The general entry of $\B$ is
$$
 b_{ef}= \delta_{e\to f} \delta_{f \ne e^{-1}} ,
$$
for $e,f \in E^{\rightarrow }$, where $E^{\rightarrow }$ is the set of bidirected
edges of $G$ (each existing edge is considered in both possible directions),
and for $e=(e_1 ,e_2)$ the reversely directed edge is denoted $e^{-1}$, so
$e^{-1} = (e_2 ,e_1 )$; further, the $e\to f$ relation means that $e_2 =f_1$
and $\delta$ is the (1-0) indicator of the
event in its lower index. Therefore, $b_{ef} =1$ exactly when $e_2 =f_1$ and
$f_2 \ne e_1$. 

With other notation (used in Section~\ref{bp}) referring to the nodes, we have
$$
 b_{l \rightarrow s, \, j \rightarrow i } = \delta_{sj } (1-\delta_{il} ) ,
$$
where $\delta$ is now the Kronecker-delta. So
$b_{l \rightarrow s, \, j \rightarrow i }  =1$ if and only if for the quadruple
in the lower indices, $l \rightarrow s=j \rightarrow i$ holds with $l\ne i$. 
Since $\B$ is a real matrix of non-negative entries,
it has a largest absolute value eigenvalue
 which is positive real, so it is the spectral radius $\rho (\B )$
 of $\B$; and $\B$ can also have some other so-called structural
real eigenvalues (those are positive in assortative networks). 
Since the characteristic polynomial of $\B$ has real coefficients, its
complex eigenvalues occur in conjugate pairs in the bulk of its spectrum.
Note that the underlying simple graph is not directed, just we consider
its edges as bidirected for the purpose of the message passing equations
of Section~\ref{bp}. 

Here we enumerate some results about the spectral properties of $\B$,
see~\cite{Glover21} for more details and proofs.
Note that the spectrum of $\B$ is not sensitive to high-degree nodes,
because a walk starting at a node cannot %turn around and
return to it
immediately. Also trees, disconnected from the graph or dangling off it,
contribute zero eigenvalues to the non-backtracking spectrum, as a
non-backtracking random walk when forced to a leaf, must stop there.
%is forced to a leaf of the tree, from where it has nowhere to go. 
%We can also show that unicyclic components yield eigenvalues
%that are 0,1, or -1. 

Some important properties of $\B$ are summarized here.
\begin{itemize}
\item
The frequently used Ihara's formula (see \cite{Glover21,Krzakala,Martin})
implies  that, whenever $G$ is not a tree $(m\ge n)$,
$\B$ has $m-n$ eigenvalues equal to 1 and 
$m-n$ eigenvalues equal to $-1$, whereas its further eigenvalues are those of 
the   $2n\times 2n$ matrix
\begin{equation}\label{K}
\K = \begin{pmatrix}
 \bm O  & \DD -\I_n \\
 -\I_n  & \A
 \end{pmatrix} , 
\end{equation}
where $\A$ is the adjacency- and $\DD =\diag (d_1 ,\dots ,d_n )$
is the diagonal degree-matrix of the
graph (it contains the row-sums of $\A$, i.e., the node-degrees $d_i$'s of $G$
in its main diagonal).

Note that $\K$ always has at least one additional eigenvalue 1, the geometric
multiplicity of which is equal to the number of the connected components
of the underlying graph (see Proposition 5.2 of~\cite{Glover21}).

\item
  For the spectral radii %largest (positive real) eigenvalues
  of $\A$ and $\B$, the relation
\begin{equation}\label{ab}
 \rho (\B )  \le \rho (\A )
\end{equation}
holds (see~\cite{Martin}).

\item If $G$ is connected with $d_{min} \ge 2$, then
 for all eigenvalues of $\B$, $|\mu | \ge 1$ holds.
 In particular, if $G$ has no nodes of degree  2, then the
 eigenvalues of $\B$ with $|\mu |=1$ are $\pm 1$'s.

\item
Proposition 5.5 of~\cite{Glover21} states that if $G$ is a connected graph which
is not a tree or cycle and
$d_{\min} \ge 2$, then $\rho (\K )> 1$.
In general (except for trees and cycles), the spectral radius $\rho (\B) >1$.
Also, by the Gershgorin theorem, for a connected graph,
$\rho (\B ) \le d_{\max }-1$ (with equality if and only if $G$ is regular).

\item
If $G$ is a connected graph that is not a cycle and $d_{\min} \ge 2$, then
$\B$ is irreducible. Therefore, the Frobenius theorem is applicable
to $\B$, and under the above conditions, it has a single positive real
eigenvalue among its maximum absolute value ones
with corresponding eigenvector of all positive real coordinates. 

 \item
Usually only the eigenvalue 0 of $\B$ is defected (its geometric and
  algebraic multiplicity is not the same), like trees, otherwise the
  eigenvectors corresponding to non-zero eigenvalues (even to
  multiple ones) are linearly independent.
  In case of random graphs, e.g., in the stochastic block model to be discussed,
  there is a bulk of the spectrum (containing $\pm 1$'s and complex conjugate
  pairs), the other
  so-called structural eigenvalues (greater than $\sqrt{c}$) are real
  (positive in the assortative case) and the corresponding eigenvectors are
  nearly orthogonal, see~\cite{Bordenave}, and the one with the largest
  absolute value, giving $\rho (\B)$ is single.

\item
  Though $\B$ is not a normal matrix, even not always diagonalizable
  (the algebraic and
geometric multiplicity of some of its eigenvalues may not be the same), 
it exhibits some symmetry. Observe that
$$
b^*_{ef} = b_{fe } = b_{e^{-1} \, f^{-1} }
$$
%$$
%b^*_{k \rightarrow l, \, j \rightarrow i } =
%b_{j \rightarrow i, \, k \rightarrow l } = b_{l \rightarrow k, \, i \rightarrow j }
%$$
for each directed pair of edges and for every entry $b^*$ of the transpose 
$\B^*$ of $\B$. This phenomenon can be described by involution and swapping.

In the sequel, the vectors are column vectors and ${}^*$ denotes the
adjoint of a matrix or a vector (in case of real entries, it is the usual
transposition).

\item
  Introduce the notation
$$
 {\breve x}_e := x_{e^{-1}} ,  \quad  e\in  E^{\rightarrow } 
$$
for relating the coordinates of the $2m$-dimensional vectors
$\x$ and $\breve \x$  of $\R^{E^{\rightarrow }}$. Now $\x$ is
partitioned  into two $m$-dimensional vectors $\x_1$ and $\x_2$, where
the coordinates of $\x_1$ correspond to $j\to i$ edges with $j<i$ and those of
$\x_2$ correspond to their inverses.
If $\x =(\x_1^* , \x_2^*)^*$, then ${\breve \x } =(\x_2^* , \x_1^*)^*$,
i.e., $\x$ and $\breve \x$ can be obtained from each other 
by swapping the set of their first $m$ and second $m$ coordinates.

Let $\V$ denote the following involution on $\R^{2m}$ ($\V =\V^{-1}$,
$\V^2 =\I$, $\V$ is an orthogonal and symmetric matrix at th same time):
$$
 \V =\begin{pmatrix}
 \bm O  & \I_m \\
 \I_m  & \bm O
\end{pmatrix} ,
$$
where the blocks are of size $m\times m$. With it, 
$\V \x = {\breve \x}$ and $\V {\breve \x }= \x$.

\item Relation to the line-graph:
  
\begin{proposition}\label{linegraph}
Let us partition the non-backtracking matrix $\B$ of the connected simple 
graph $G$ on $n$ nodes and $m$ edges into the following four $m\times m$
quadrants:
$$
 \B =\begin{pmatrix}
 \B_{11}  & \B_{12} \\
 \B_{21}  & \B_{22} 
\end{pmatrix} ,
$$
where the first $m$ rows (columns) correspond to the $j\to i$ edges
with $j<i$, and the next $m$ rows (columns) to their inverses,
in the same order. Then
\begin{equation}\label{B}
\B_{11}^* =\B_{22}, \quad \B_{22}^* =\B_{11}, \quad \B_{12}^* =\B_{12},
\quad \textrm{and} \quad \B_{21}^* =\B_{21}.
\end{equation}
Further, $\B_{11} + \B_{12} + \B_{21} + \B_{22}$ is equal to the 
$m\times m$ adjacency matrix of the line-graph of $G$.
\end{proposition}  

Note that the nodes of the line-graph of $G$ are its edges, and two  edges are 
adjacent if and only if they have a node in common.
Let $\N$ denote the $n\times m$ node--edge (0-1) incidence matrix of
the simple graph $G$ (except for trees, $n\le m$). Then $\N \N^*$ and
$\N^* \N$ are Gramians (covariance matrices), so they are positive
semidefinite, have the same rank and the same set of positive eigenvalues. 
$\N \N^* =\DD + \A$ is called signless Laplacian (whereas, the Laplacian
matrix is $\DD -\A$, see~\cite{Bolla13}). On the other hand,
$\N^* \N - 2 \I_m$ is the adjacency matrix of the line-graph of $G$
(because the valence of each edge is 2). Consequently, its 
eigenvalues are those of the positive eigenvalues of $\DD +\A$ decreased by 2,
and the others
are equal to -2, see~\cite{Haemers} for more explanation.
Furthermore (see, e.g.,~\cite{Lovasz}) 
it is known that if the line-graphs of two simple
graphs, provided they both have node-degrees at least 4,  are isomorphic,
then they
are isomorphic too. However, if the degree condition does not hold,
it can happen that two not isomorphic graphs have isomorphic line-graphs.
For example, if we have a triangle and a star on 4 vertices, then the
adjacency matrix of their line-graphs is the same $3\times 3$ matrix
(adjacency matrix of a complete graph, as any two edges of them have
a node in common). The same happens if we have graphs with such
dangling triangles or stars. %and in case of sparse graphs.
Therefore, in case of sparse graphs, $\B$  carries more information for the
graph  than  its line-graph; further,
considering its  edges in both
directions makes it possible to perform non-backtracking random walks and
message passing along them, see Section~\ref{bp}.

%%%%%%%%%%%%%%%%%% new

In~\cite{Mulas} it is proved that two simple graphs are isomorphic if and
only if their corresponding non-backtracking graphs are isomorphic
(their non-backtracking matrices are the same if we consider the
bidirected edges
in the same order). Under non-backtracking graph we understand the directed
graph on
$2m$ vertices with adjacency relation corresponding to the definition of the
non-backtracking matrix.

%%%%%%%%%%%%%%%%%% end new

\begin{proof} (Proposition~\ref{linegraph})
  
Let us label the rows (and columns) of $\B$ with the edges, ordered 
lexicographically: the first $m$ rows correspond to the $j\to i$ edges 
$(1\le j<i \le n)$, 
 whereas the second $m$ ones to the inverse edges $i\to j$, in the same order.
Obviously, $\B_{12}$ is a symmetric matrix as its corresponding upper- and
lower-diagonal entries, producing an entry 1, contain the
matching %$e\to f$ and $f\to e$
edges in both directions;
%(here $e\to f$ is the
%shorthand for the fact that the end-node of $e$ is the start-node of $f$, and
%the start-node of $e$ is different of the end-node of $f$); 
$\B_{21}$ is a symmetric matrix for the same reason.
$\B_{11}$ is upper-, $\B_{22}$ is lower-triangular, and they are
transposes of each other, because there is a one-to one correspondence
between the upper-diagonal entries equal to 1 of $\B_{11}$ and the 
lower-diagonal entries equal to 1 of $\B_{22}$ such that one contains an 
$e\to f$ and  the other an $f\to e$ relation. The other entries are equal to 0. 

The above argument shows that in the same position, at most only one entry can 
be equal to 1 in the four quadrants of $\B$, which means that the
corresponding adjacency entry of the line-graph of $G$ is 1; 
the others are zeros. 
It is also obvious that the undirected $\{ l,j \}$ and $\{ j,i \}$ edges can
be joined at $j$ in four possible ways: 
$i \to j\to l$, $l \rightarrow j \leftarrow i$, 
$i \leftarrow j\rightarrow l$, and $l \to j\to i$. Only one of them
results in the true relation $\{ l\to j \} \to \{ j \to i\}$.
This proves the second statement. The sum of the four quadrants of $\B$ 
is of course symmetric, because of the first statement.
\end{proof}

As a consequence of Proposition~\ref{linegraph},
$$
 (\B^{\ell })^* \V = \V \B^{\ell} 
 $$
 holds true for any natural number $\ell$. Applying this for $\ell =1$,
 the relation $\B^* = \V \B \V$ and 
 $$
 \B^* {\breve \x} = \breve{(\B \x)} 
 $$
 is valid for any vector $\x \in \R^{2m}$.
This phenomenon is called PT (parity-time) invariance is physics.
 
This implies the following: if $\x$ is a right eigenvector of $\B$ with
a real eigenvalue $\mu$, then $\breve \x$ is a right eigenvector of $\B^*$
with the
 same eigenvalue. Consequently, if $\x$ is a right eigenvector of $\B$, then
 $\breve \x$ is a left eigenvector of it (and vice versa), with the same
 real eigenvalue.
% This usually does not hold for the right and left eigenvectors of  $\K$.

 Another easy implication is that $(\B \V )^* =\V \B^* =\V \V \B \V =\B \V$,
 so $\B \V$  and $\V \B$ are 
 symmetric matrices that also follows by~\eqref{B}. Indeed,
 $$
 \B \V =\begin{pmatrix}
 \B_{12}  & \B_{11} \\
 \B_{22}  & \B_{21} 
\end{pmatrix} ,
\quad
 \V \B  =\begin{pmatrix}
 \B_{21}  & \B_{22} \\
 \B_{11}  & \B_{12} 
\end{pmatrix} .
$$

Consequently, $\B \V$  is diagonalizable in an orthogonal basis, and so,
the singular values of $\B$ are the absolute values of its real eigenvalues,
whereas, the singular vector pairs are the eigenvectors and their swappings.
Namely, the leading singular values are $d_i -1$ for $i=1,\dots ,n$, and so,
they are completely determined by the degree sequence of $G$. However,
the eigenvalues of $\B$ are quite different.

%for any integer $\ell \ge 0$,
%$$
% \B^{\ell } \V = \V \B^{\ell} .
%$$
%So $\B^{\ell } \V$ is a symmetric matrix with eigenvalues $\sigma_{j,\ell}$ and
%orthonormal set of eigenvectors $\x_{j,\ell}$ for $1\le j  \le 2m$. From here,
%$$
%  \B^{\ell } =\sum_{j=1}^{2m} \sigma_{j,\ell} \x_{j,\ell} {\breve \x}_{j,\ell} %.
%$$
%Since $\V$ is also an orthogonal matrix, the vectors  ${\breve \x}_{j,\ell}$
%also form an orthonormal basis in  $\R^{E^{\rightarrow }}$.
%This gives rise to the SVD of  $\B^{\ell }$, see~\cite{Bordenave}.

\item
  But what is the relation between the eigenvectors corresponding to the 
leading (real) eigenvalues of $\B$ and $\K$?
  To ease the discussion, two auxiliary matrices,
  defined in~\cite{Stephan} will be used:
  the $2m\times n$ \textit{end matrix} $\bm{End}$ has entries
  $end_{ei} =1$ if $i$ is the end-node of the (directed) edge $e$ and 0,
  otherwise;
  the $2m\times n$ \textit{start matrix} $\bm{Start}$ has entries
  $start_{ei} =1$ if $i$ is the start-node of the (directed) edge $e$ and 0,
  otherwise. Then for any vector $\uuu \in \R^n$ and for any edge
  $e=\{ i\to j \}$, the following holds:
  $$
  ( \bm{End}\, \uuu )_e = u_j \quad \textrm{and} \quad
  ( \bm{Start}\, \uuu )_e =u_i .
  $$
  Consequently,  $\bm{End} \, \uuu$ is the $2m$-dimensional inflated version
  of the $n$-dimensional vector $\uuu$, where the coordinate $u_j$ of $\uuu$
  is repeated as many times, as many edges have end-node $j$; likewise,
 in the $2m$-dimensional inflated vector $\bm{Start} \, \uuu$, the coordinate 
  $u_i$ of $\uuu$
  is repeated as many times, as many edges have start-node $i$.
  As each edge is considered in both possible directions,
  this numbers are the node-degrees $d_j$ and $d_i$, respectively.
 % Note that if for some $k<n$, the coordinates of $\uuu$ form $k$
 % well-separated clusters, then so do the coordinates of the inflated vector.
 Note that 
\begin{equation}\label{es}
\bm{End}^* \, \bm{End}=\bm{Start}^* \, \bm{Start} =
 \diag (d_1 ,\dots ,d_n ) =\DD .
\end{equation}

Akin to~\cite{Krzakala},
%the following relation between
%  the eigenvectors corresponding to the leading eigenvalues of $\B$ (the same
%  as those of $\K$) is proved.
for any vector  $\x \in \R^{2m}$  define
 $$
  x_i^{out} := \sum_{j: \, j\sim i} x_{i\to j}  \quad \textrm{and}  \quad
  x_i^{in} := \sum_{j: \, j\sim i} x_{j\to i}, \quad i=1,\dots ,n.
$$
We put these coordinates into the $n$-dimensional
(column) vectors $\x^{in}$ and $\x^{out} $.
Trivially,
\begin{equation}\label{inout}
  \x^{out} =\bm{Start}^* \x, \quad \textrm{and} \quad
  \x^{in} = \bm{End}^* \x , \quad i=1,\dots ,n.
\end{equation}

For them, we have the following relations (now more comfortably, we use
$\B^*$ instead of $\B$):
%with the shorthand $f\to e$ for
%the fact that $f_2 =e_1$.
\begin{equation}\label{rel1}
\begin{aligned}
  (\B^* \x)^{out}_i &= \sum_{e: \, e_1 =i}  (\B^* \x)_e = 
 \sum_{e: \, e_1 =i} \sum_{f\to e , \, f \ne e^{-1}}  x_f \\
&= \sum_{e: \, e_1 =i}  [\sum_{f\to e}  x_f -  x_{e^{-1}} ]  \\
&=\sum_{f: \, f_2 =i} x_f \sum_{e: \, e_1 =i} 1 - 
   \sum_{e: \, e_1 =i}  x_{e^{-1}}  \\
&= x_i^{in} d_i -\sum_{e: \, e^{-1}_2 =i}  x_{e^{-1}} = 
    d_i x_i^{in} -x_i^{in}= (d_i -1 ) x_i^{in} .
\end{aligned}
\end{equation}
Also,
\begin{equation}\label{rel2}
\begin{aligned}
  (\B^* \x)^{in}_i &= \sum_{e: \, e_2 =i}  (\B^* \x)_e = 
 \sum_{e: \, e_2 =i}  \sum_{f\to e , \, f \ne e^{-1}}  x_f \\
&=\sum_{j=1}^n a_{ji} \sum_{f: \, f_2 =j, \, f_1 \ne i}  x_f  \\
&=\sum_{j=1}^n a_{ji} \sum_{f: \, f_2 =j}  x_f 
   -\sum_{j=1}^n a_{ji} x_{i\to j}\\
   &=\sum_{j=1}^n a_{ij} x^{in}_j -  \sum_{j: \, j\sim i} x_{i\to j}
 = (\A \x^{in})_i -x^{out}_i ,
\end{aligned}
\end{equation}
where we used that the (0-1) adjacency matrix $\A$ of the graph is symmetric
with entries $a_{ij} =a_{ji} =\delta_{i\sim j}$.

Summarizing, if $\x$ is a right eigenvector of $\B^*$ with (real) eigenvalue
$\mu$, then
\begin{equation}\label{uj}
\mu \begin{pmatrix}
 \x^{out} \\
 \x^{in}
\end{pmatrix} =
 \begin{pmatrix}
 (\B^* \x)^{out} \\
 (\B^* \x)^{in}
\end{pmatrix} = 
 \begin{pmatrix}
 \bm O  & \DD -\I_n  \\
 \bm -\I_n  & \A
\end{pmatrix} 
\begin{pmatrix}
 \x^{out} \\
 \x^{in}
\end{pmatrix} 
\end{equation}
%because each incoming edge $j\to i$ contributes $d_i -1$ times to the
%%outgoing edges of $i$; likewise, each incoming edge $k\to j$
%with $k\ne i$ contributes to the incoming edge $j\to i$.
and
$$
 \begin{pmatrix}
 (\B^* \x)^{out} \\
 (\B^* \x)^{in}
\end{pmatrix} = 
\begin{pmatrix}
 (\DD -\I_n ) \x^{in} \\
 \A \x^{in} -\x^{out}
 \end{pmatrix} =
  \K \begin{pmatrix}
 \x^{out} \\
 \x^{in}
 \end{pmatrix}  .
 $$
In particular, if $\x$ is a right eigenvector of $\B^*$ with a real eigenvalue
$\mu \ne 0$,  then the
$2n$-dimensional vector comprised of parts $\x^{out}$ and $\x^{in}$ is a 
right eigenvector of
$\K$ with the same eigenvalue. Indeed,
$$
\mu \begin{pmatrix}
 \x^{out} \\
 \x^{in}
\end{pmatrix} =
\begin{pmatrix}
  (\mu \x)^{out} \\
  (\mu \x)^{in}
 \end{pmatrix} =
 \begin{pmatrix}
 (\B^* \x)^{out} \\
 (\B^* \x)^{in}
\end{pmatrix} =
\K \begin{pmatrix}
 \x^{out} \\
 \x^{in}
\end{pmatrix} .
$$
According to the previous remarks, the vector $\x$
is a left eigenvector, and $\breve \x$ is a right eigenvector of
$\B$ with the same eigenvalue.
To both of them the two segments, $\x^{out}$ and $\x^{in}$ of the right
eigenvector of $\K$ are responsible.
In view of the relation $ \x^{out} =\frac1{\mu} (\DD -\I_n ) \x^{in}$,
obtained by~\eqref{uj},
it suffices to consider only $\x^{in} \in \R^n$ for further clustering purposes,
see Section~\ref{id}.

 \end{itemize}

\subsection{Random graphs obtained by bond-percolation}\label{SBM}

In the simple graph  scenario, in~\cite{Newman22}, it is proved that
the so-called message passing (in other words, belief propagation) system 
of equations  (briefly, BP) has a non-trivial solution
if and only if, for the edge retention probability $\beta$,
$\beta > \frac1{\rho (\B )}$ holds, where $\rho (\B )$ is the
spectral radius of the
non-backtracking matrix $\B$ of a given sparse social network.
So $\frac1{\rho (\B )}$ 
is the \textit{bond-percolation threshold} for the giant component
to appear in this simple case;
which in view of~\eqref{ab} is greater than $\frac1{\rho (\A )}$ 
of~\cite{Bollobas} in the dense case. 

Now our purpose is to characterize the non-backtracking spectrum of a random 
graph  $G_n$ on $n$ nodes
coming from the $SBM_k$ model, and via perturbations, try to use
these properties to find out the number of clusters and the clusters themselves.
The parameters of the sparse $SBM_k$ model are as follows.

The $k\times k$
\textit{probability matrix} $\PP$ of the percolated sparse random graph 
$G_n \in SBM_k$  has entries
$$
 p_{ab} =\frac{c_{ab}}{n} ,
$$
where the $k \times k$
symmetric \textit{affinity matrix} $\C =(c_{ab})$ stays constant as 
$n\to\infty$. An edge between $i<j$ comes into existence,
independently of the others, with probability
$p_{ab}$ if $i\in V_a$ and $j\in V_b$, where $(V_1 ,\dots ,V_k)$ is a partition
of the node-set $V$ into $k$ disjoint clusters.
This will produce the upper-diagonal
part of the $n\times n$ random adjacency matrix $\A$, and $a_{ji} :=a_{ij}$.
It can be extended to the  $i=j$ case
when self-loops are allowed, or else, the diagonal entries of the
adjacency matrix are zeros.

Let $\bar \A$ denote the $n\times n$ inflated matrix of the $k\times k$ 
matrix $\PP$: ${\bar a}_{ij} =p_{ab}$ if $i\in V_a$ and $b\in V_b$.
When loops are allowed, then $\E (a_{ij} ) ={\bar a}_{ij} $ for all
$1 \le i,j \le n$.
In the loopless case, the expected adjacency matrix $\E \A$ differs from
$\bar \A$  with respect to the the main diagonal,
but the diagonal entries are negligible as will be shown in Section~\ref{id}. 

Specifically, sometimes $c_{ab} =c_{in}$ is the within-cluster $(a=b)$
and $c_{ab} =c_{out}$ is the between-cluster $(a\ne b)$ affinity.
In~\cite{Newman10}, the network is called \textit{assortative} if
$c_{in} > c_{out}$,
and \textit{disassortative} if $c_{in} < c_{out}$. Of course, remarkable
difference is needed between the two, to recognize the clusters. 
The cluster sizes are $n_1 ,\dots , n_k $ ($\sum_{i=1}^k n_i =n$), so the
$k\times k$ diagonal matrix  $\RR :=\diag (r_1 ,\dots ,r_k)$, where 
$r_a =\frac{n_a}{n}$ is the relative size of cluster
$a$ $(a=1,\dots ,k)$,  is also a model parameter ($\sum_{a=1}^k r_a =1$). 
The model $SBM_k$ is called \textit{symmetric} if
$r_1 =\dots =r_k =\frac1{k}$ and all diagonal entries of the affinity matrix
are equal to $c_{in}$, whereas the off-diagonal ones to $c_{out}$.

The average degree of a real world graph on $m$ edges and $n$ nodes is
$\frac{2m}{n}$. 
The expected average 
degree of the random graph $G_n$ generated from the $SBM_k$ model is
\begin{equation}\label{deg}
  c=  \frac1{n} \sum_{i=1}^n \sum_{j=1}^n {\bar a}_{ij} = 
  \frac1{n} \sum_{a=1}^k \sum_{b=1}^k n_a n_b p_{ab} =
\frac1{n^2} \sum_{a=1}^k \sum_{b=1}^k n_a n_b c_{ab} =
%\sum_{a=1}^k \sum_{b=1}^k r_a r_b c_{ab} =
 \sum_{a=1}^k r_a  c_a ,
\end{equation}
where $c_a =\sum_{b=1}^k r_b c_{ab}$ is the average degree of cluster $a$.
It is valid only if self-loops are allowed. Otherwise, $c_a$ and $c$
should be decreased with a term of order $\frac1{n}$, but it will not make
too much difference in the subsequent calculations.

In~\cite{Bordenave}, the case when  $c_a =c$ for all $a$ is considered.
In this case $\frac1{c} {\bar \A}$ is a stochastic matrix, and so, the
spectral radius of $\bar \A$ is $c$.

The symmetric case is a further special case of this,
when %$r_1 =\dots =r_k =\frac1{k}$ 
%and $c_{ab} =c_{in}$ for $a=b$ and $c_{out}$, otherwise.
%In the symmetric case,
$$
 c=\frac{c_{in}+(k-1) c_{out}}{k} . 
$$
In this case, the separation of the clusters
only depends on the  $c_{in}$, $c_{out}$ relation.
If $c_{in}$ is ``close'' to $c_{out}$, then the groups cannot be distinguished. 
The detectability threshold (Kesten--Stigum threshold) in the symmetric case is 
\begin{equation}\label{2}
 | c_{in} - c_{out} | >k\sqrt{c} ,
\end{equation}
see~\cite{Bordenave,Decelle,Moore}.

In the symmetric $k=2$ case, the authors of~\cite{Mossel} prove that
under condition~\eqref{2}, their algorithm gives a clustering that has a
correlation with the true one separated from zero; furthermore, this correlation
tends to 1 as the ratio of the left and right hand sides of~\eqref{2} 
increases. They also write that the $k\ge 2$ and 
$c_1 =\dots =c_k$ case is the hardest one,
as otherwise the clusters could be distinguished by sorting the node-degrees.

Note that Theorem 4 of~\cite{Bordenave} allows a small fluctuation for 
$n_a$'s, and so,
for $r_a$'s and $c_a$'s too. We make use of this when compare the
structural eigenvalues of the non-backtracking matrix to the eigenvalues of the
so-called transition matrix assessed during the BP process.
Observe that these assumptions are natural in social networks: the relative
sizes of communities are balanced (tend to a limit as $n\to\infty$); further,
the average degrees of nodes of the clusters do not differ much as
$n\to\infty$, which means that members of different communities have on average
closely the same number of bonds to members of their own or other communities.

Stephan and Massouli\'e~\cite{Stephan} investigate more general random graph
models, also in the edge-weighted case. In the special unweighted %$SBM_k$
situation we shall use the following (informal) statement of them
for the $SBM_k$ model, where the matrix $\bar \A$ is not only the
diagonal-corrected expected adjacency matrix, but it is approximately the
variance matrix of the random adjacency  matrix $\A$.
\begin{proposition}[based on Theorem 1 of~\cite{Stephan}]\label{propo}
 Assume that $k =\rk {\bar \A} =n^{o(1)}$, the graph  is sparse enough, and the
  eigenvectors corresponding to the non-zero eigenvalues of the matrix
  $\bar \A$  are sufficiently delocalized.
  Let $k_0$ denote the number of eigenvalues of $\bar \A$
  whose absolute value is larger than $\sqrt{\rho}$, where $\rho$ is
  the spectral radius of $\bar \A$: these are
  $\nu_1 \ge \dots \ge \nu_{k_0}$ with corresponding eigenvectors
  $\uuu_1 ,\dots ,\uuu_{k_0}$ (they form an orthonormal system as $\bar \A$
  is a real symmetric matrix).
  Then for $i\le k_0 \le k$, the $i$th largest eigenvalue $\mu_i$
  of $\B$ is asymptotically
  (as $n\to\infty$) equals to $\nu_i$ and all the other eigenvalues of $\B$
  are constrained to the circle (in the complex plane)
  of center $0$ and radius $\sqrt{\rho}$.
  Further, if $i\le k_0$ is such that $\nu_i$ is a sufficiently isolated
  eigenvalue of $\bar \A$,
  then the standardized eigenvector of $\B$  corresponding to
  $\mu_i$ has inner product close to 1
  with the standardized inflated version of $\uuu_i$, namely, with
  $\frac{\bm{End} \, \uuu_i }{\| \bm{End} \, \uuu_i \| }$.
\end{proposition}
Note that if our random graph comes from the $SBM_k$ model with
$k\times k$ parameter matrices $\RR$ and $\C$, then $\bar \A$
has rank at most $k$, and its nonzero eigenvalues are identical to the
real eigenvalues of the matrix $\RR \C$ (same as those of
$\RR^{\frac12}  \C \RR^{\frac12}$),
and the eigenvectors of it are inflated versions
of those of $\RR^{\frac12}  \C \RR^{\frac12}$; consequently, they are
step-vectors on $k$ different  steps and so, 
the k-means algorithm is applicable to the $k_0$-dimensional representatives 
of the nodes constructed with the first $k_0$ shrunken,
normalized eigenvectors of $\B$
corresponding to its leading eigenvalues $\mu_1 ,\dots ,\mu_{k_0}$. 
The point is that we do not need
the model parameters and the spectral decomposition of $\bar \A$ itself
for the k-means clustering.
This issue together with inflation--deflation techniques will be discussed 
in details  in Section~\ref{id}.
Also note that in the $SBM_k$ model, when $c_a =c$ $(a=1,\dots ,k )$,
then $\rho=c$ and $k_0 =k$.

\section{Belief propagation (BP) in the general SBM}\label{bp}

To derive the percolation threshold analytically, we use the  BP
or cavity method of~\cite{Moore} in the generic case.
In~\cite{Moore}, it was developed for the symmetric case.
Akin to  the BP, called message passing in~\cite{Newman22}, 
let us introduce the following notation: for $a=1,\dots ,k$,
$$
 \psi_i^a \propto \PPP (i \textrm{ is in the cluster } a) ,
$$
given the observed graph on $n$ nodes, coming from the $SBM_k$ model, where 
$\psi_i^a$ for $a=1,\dots ,k$ defines the marginal membership (state) 
distribution of node $i$.

We assume that our neighbors are independent of each other, when conditioned
on our own state. Equivalently, we assume that our neighbors are correlated
only through us. This can be modeled by having each node $j$ send a message
$\psi_{j \to i }$ to $i$, which is an estimate of $j$'s marginal
(in other words, state or membership) if $i$ were not there
(more precisely, if we did not know whether or not there is an edge
between $i$ and $j$). Therefore, the conditional probability
$$
 \psi_{j \to i}^a := 
 \PPP (j \textrm{ is in cluster } a \textrm{ when } i \textrm{ is not present})
$$
(or when we do not know the membership of $i$) can be computed through the 
neighbors of $j$ that are different from $i$ (in the realization of the
random graph coming from the $SBM_k$ model) as follows:
\begin{equation}\label{psik}
\psi_{j \to i}^a = C_a^{ij} r_a \prod_{l\sim j , \, l\ne i} 
 \sum_{b=1}^k \psi_{l \to j}^b \, p_{ab}  , \quad a=1,\dots ,k ,
\end{equation}
where $C_a^{ij}$ is a normalizing factor.
Here $\psi_{l \to j}^b$ is the conditional probability that node $l$
(which is not $i$ and to which $j$ is connected) belongs to cluster $b$ even in
the absence of $j$, and the summation is an application of the rule of
completely disjoint events with probabilities $p_{ab}$ for $b=1,\dots ,k$.

This so-called BP or \textit{message-passing equation}
is a system of $2mk$ non-linear equations with the same number of unknowns.
It can be solved by initializing messages randomly, then repeatedly updating
them via~\eqref{psik}.
This procedure usually converges quickly 
and the resulting fixed point 
gives a good estimate of the marginals:
$$
 \psi_i^a \propto r_a \prod_{j\sim i } \sum_{b=1}^k 
 \psi_{j \to i}^b \, p_{ab} ,
$$
where the constant of proportionality is chosen according to
$\sum_{a=1}^k \psi_{i}^a =1$. 
However, the system of equations contains the model
parameters, so it can be solved only if the model parameters are known.

In \cite{Moore}, the symmetric case is treated,
when BP has a trivial fixed point 
$\psi_{j \to i}^a =\frac1{k}$, for $a=1,\dots ,k$. If it gets stuck there,
then BP does no better than chance. It happens when this
trivial fixed point of this discrete dynamical system is asymptotically stable.

In the generic case, we want to have an unstable fixed point via linearization:
$$
 \psi_{j \to i}^a := r_a + \ep_{j \to i}^a .
$$ 
If we substitute this in~\eqref{psik} and expand to first order in $\epv$
(vector of $2mk$ coordinates $\ep_{j \to i}^a$'s), the updated system of
equations is
\begin{equation}\label{hosszu}
\begin{aligned}
 \ep_{j \to i}^a &= \psi_{j \to i}^a  -  r_a = r_a \left\{ C_a^{ij}
 \prod_{l\sim j , \, l\ne i} \left[ \sum_{b=1}^k \psi_{l \to j}^b \, 
 p_{ab} \right] -1 \right\}  \\
 &= r_a \left\{ C_a^{ij} \prod_{l\sim j , \, l\ne i} \left[ \sum_{b=1}^k (r_b +
 \ep_{l \to j}^b ) \, p_{ab} \right] -1 \right\} \\
&= r_a \left\{ C_a^{ij} \prod_{l\sim j , \, l\ne i} \left[ \sum_{b=1}^k r_b p_{ab}   +\sum_{b=1}^k \ep_{l \to j}^b \, p_{ab} \right] -1 \right\} \\
&= r_a \left\{ C_a^{ij} \prod_{l\sim j , \, l\ne i} \left[ \frac{c_a}{n} 
  +\sum_{b=1}^k \ep_{l \to j}^b \, \frac{c_{ab}}{n} \right] -1 \right\} \\
 &= r_a \left\{ C_a^{ij} (\frac1{n} )^{s_j -1 } \left[ \sum_{b=1}^k 
 \sum_{l\sim j , \, l\ne i}
 \ep_{l \to j}^b \, c_{ab} \, c_a^{s_j -2} +c_a^{s_j -1} \right] -1 
  \right\} +O (\epv^2 ) ,
\end{aligned}
\end{equation}
where $s_j$ denotes the number of neighbors of $j$ and $s_j-1$ is the number of
its neighbors that are different from 
$i$ (this number is frequently 0 or 1, as we have a sparse graph).
If $s_j <2$ happens, then the corresponding entry of the non-backtracking
matrix is 0. % as it will be seen soon.
To specify the normalizing factor $C_a^{ij}$, we substitute zeros for $\ep$'s
that provide the trivial solution. This approximately yields
$$
 C_a^{ij} \left( \frac1{n} \right)^{s_j -1}  \, c_a^{s_j -1}  -1 =0  ,
$$
so
$$
 C_a^{ij} = \left( \frac{n}{c_a} \right)^{s_j -1} .
$$ 
Substituting this into the last equation of~\eqref{hosszu}, we get
$$  
\begin{aligned}
\ep_{j \to i}^a 
 &= r_a \left\{ \left( \frac{n}{c_a} \right)^{s_j -1 } 
 \left( \frac1{n} \right)^{s_j -1} c_a^{s_j -2 } 
 \left[ \sum_{b=1}^k \sum_{l\sim j, l\ne i} \ep_{l \to j}^b \, c_{ab} 
 \,   +c_a \right] -1  \right\} +O (\epv^2 )  \\
&= r_a \left\{ \frac{1}{c_a} 
 \left[ \sum_{b=1}^k \sum_{l\sim j, l\ne i} \ep_{l \to j}^b \, c_{ab} \,  
 +c_a \right] -1  \right\} +O (\epv^2 ) .
\end{aligned}
$$
The linear approximation of the above system of equations is
\begin{equation}\label{lin}
 \epv = \M  \epv ,
\end{equation}
where $\epv$ is a $2mk$-dimensional vector and 
the $(j \to i, \, a ) \, , \, ( l \to s ,\, b )$ entry of 
$\M$ is $\partial \ep_{j \to i}^a / \partial \ep_{l \to s}^b $.
It is not zero only for $j \to i $, $l \to s$ pairs with $s=j$
and $l\ne i$, i.e., when the corresponding entry of $\B$ is 1.
If it is not zero, then the partial derivative of $\ep_{j \to i}^a$ with
respect to $\ep_{l \to j}^b$ is $\frac{r_a}{c_a} c_{ab}$ (the second
order terms are disregarded).
Therefore, it can easily be seen that $\M$ is a Kronecker-product: 
$$
  \M =\B \otimes \T ,
$$ 
i.e., $\T$ is substituted for each entry 1 of the non-backtracking matrix $\B$,
where $\T = \G \RR \C$ is the \textit{transmission matrix} 
with $\G =\diag (\frac1{c_1} ,\dots , \frac1{c_k} )$.  
Note that 
$\T$ is a stochastic matrix (so its largest eigenvalue is 1) only if
$r_1 =\dots =r_k$.

For the solution of~\eqref{lin}, the fixed point of the linear
dynamical system
$$
 \epv^{(t+1)} = \M  \epv^{(t)} , \quad t=1,2,\dots
$$
is looked for. We can find a fixed point other than the trivial $\0$ if
$\0$ is an unstable solution, for which a sufficient condition is that
the spectral radius of the matrix $\M$ is greater than 1.
In this way, we have proved the following.

\begin{theorem}\label{bpgeneral} 
  Assume that the graph $G$ comes from the $SBM_k$ model with
  positive integer $k$ and $k\times k$ parameter
matrices $\RR =\diag (r_1 ,\dots ,r_k )$ (cluster proportions) and $\C$
(symmetric affinity matrix). The approximating linear dynamical system  
behind the BP system~\eqref{hosszu} is $\epv^{(t+1)} = \M  \epv^{(t)}$,
where $\epv$ is a $2mk$-dimensional vector and 
the $2mk\times 2mk$ matrix of the  system is $\M =\B \otimes \T$.
Here $\B$ is the non-backtracking matrix of the graph  $G$ and 
$\T = \G \RR \C$ is the transmission matrix
with $\G =\diag (\frac1{c_1} ,\dots , \frac1{c_k} )$, where
$c_a =\sum_{b=1}^k r_b c_{ab}$ is the average degree of cluster $a$,
for $a=1,\dots ,k$.
If the trivial fixed point $\0$ is an unstable fixed point of the above
system, then there is at least one eigenvalue of the matrix
$\B \otimes \T$ (product of eigenvalues of $\B$ and $\T$) 
that is greater than 1 (in absolute value).
\end{theorem}

Note that in~\cite{Moore} only the symmetric case is treated.
In~\cite{Bordenave} the special case $c_1 =\dots =c_k =c$ is considered
when the matrix $\T$ becomes  $\T = \frac1{c} \RR \C$; then 
the leading eigenvalues of $\B$ and $\RR \C$ are w.h.p. ``close'' to each other.
Also, the largest eigenvalue of $\RR \C$ is $c$, which is trivially the case if
$k=1$ and we have the Erd\H os--R\'enyi random graph.
More precisely, the authors of~\cite{Bordenave} allow ``small'' fluctuations
of the cluster membership proportions that causes the same order of
fluctuations in the average degrees of the clusters. For the 
membership proportion of cluster $a$, denoted by $r_a^{(n)}$, the relation
$\lim_{n\to\infty} r_a^{(n)} =r_a$ is assumed with $\sum_{a=1}^k r_a =1$.
Specifically, in~\cite{Bordenave},
the assumption
\begin{equation}\label{ra}
 \max_{a\in \{ 1,\dots ,k \} } | r_a^{(n)} -r_a | =  O (n^{-\gamma} ) 
\end{equation}
is made  with some $\gamma  \in (0,1]$,
where $n$ denotes the number of nodes.
This assumption implies that in the $c_1 =\dots =c_k =c$ case,
$\max_{a\in \{ 1,\dots ,k \} } | c_a^{(n)} -c | =  O (n^{-\gamma} )$.

Theorem 4 of~\cite{Bordenave} states that if 
\begin{equation}\label{maxc}
  \max_a c_a^{(n)} =c + O (n^{-\gamma} ), 
\end{equation}
with some $\gamma \in (0,1]$, and the
 relative proportions of the clusters converge, then  w.h.p.
$$
 \mu_i =\nu_i + o(1) \quad (i=1,\dots k_0 ) \quad \textrm{and} \quad
 \mu_{i} < \sqrt{c} +o(1) \quad (i>k_0 ) ,
$$
where $\mu_i$'s and $\nu_i$'s $(i=1,\dots ,k_0 )$ are the structural 
eigenvalues of $\B$ and $\RR \C$, respectively, whereas $k_0 \le k$ is the
positive  integer for
which $\nu_i^2 \ge \nu_1$ $(i=1,\dots k_0 )$ and $\nu_{k_0 +1}^2 < \nu_1$ holds.
Therefore, in the $SBM_1$ (Erd\H os--R\'enyi) model,  
$\mu_1 =c+o(1)$ and $\mu_2 \le \sqrt{c} +o(1)$.

However, in the generic case too, 
simulations show that sometimes the leading eigenvalues of $\B$ and
$\RR \C$ are closer to each other than those of $\B$ and $c \T$,
but always closer than those of $\A$ and any of them.
The following considerations can be made.
If~\eqref{maxc} holds, then 
\begin{equation}\label{minc}
  \min_a \frac{c}{c_a^{(n)} }  = \frac{c}{ \max_a 
  c_a^{(n)} } = 1 + O (n^{-\gamma} ) . 
\end{equation}
By the rules of the Kronecker-products, 
$$
  \B \otimes c\T  = (\I_{2m} \B ) \otimes (c\G \RR \C ) = 
(\I_{2m} \otimes c\G ) ( \B \otimes \RR \C ) .
$$
But 
$$
  \I_{2m} \otimes c\G =\diag (\frac{c}{c_1} , \dots .\frac{c}{c_k},
 \dots , \frac{c}{c_1} , \dots .\frac{c}{c_k} ) 
$$ 
repeated $2m$ times.
By minimax theorems, the eigenvalues of  
$(\I_{2m} \otimes c\G ) (\B \otimes \RR \C )$ are 
within factors $u$ and $v$ of those of $\B \otimes \RR \C$, where 
$$
 u= \min_a \frac{c}{c_a^{(n)}} \quad \textrm{and} \quad  
 v= \max_a \frac{c}{c_a^{(n)}} .
$$
Here $u$ and $v$ depend on $n$, but we do not denote this dependence.
Also, for fixed ``large'' $n$, the eigenvalues of  $c\G \RR \C $ are 
within a factor $u$ and $v$ of those of $\RR \C$.

Then the condition  $\lambda (\B) \, \lambda (\T )>1$, i.e.,  
$$
 \lambda (\B) \, \lambda (c\T ) =\lambda (\B) \, \lambda (c \G \RR \C )  >c 
$$
holds if
$$
  \lambda (\B) \, (u \lambda (\B) )   > c . 
$$
This implies that
$$
  \lambda (\B) \ge \frac{\sqrt{c}}{\sqrt{u}} \quad \textrm{and} \quad
  \lambda (c\G \RR \C) \ge u \lambda (\B ) \ge \sqrt{u} \sqrt{c}
$$
should hold, where $u\le 1$. 

Consequently, real eigenvalues of $\B$ that are larger than $\sqrt{c}$
% or those of $c \G \RR \C$ that are a bit smaller than $\sqrt{c}$
should be taken into account. %This number is $k_0$.
However, $\RR$ and $\C$ contain the unknown model
parameters, $\G$ depends on them, and there is only $\B$ at our disposal
from a sample. 
But the model itself is adapted to the data, i.e., the underlying graph.
In~\cite{Newman22}, the author makes the following crucial observation:
``The calculation, we have described is based on a fit of the stochastic
block model to a network that was itself generated from the same model.''

So we should consider the
real eigenvalues of $\B$ that are  greater than $\sqrt{c}$. 
This number will be denoted by $k_0$.
Also, the non-zero eigenvalues of the expected adjacency matrix $\bar \A$ 
of the  next section are the same as
those of $\RR \C$, so they are in the neighborhood of the leading eigenvalues 
of $\B$ within a factor between $u$ and $v$. On the contrary,
their closeness to $\A$
is up to an additive constant (less than $\sqrt{c}$) 
as will be shown in Section~\ref{id}. 

Note that when the average degrees $c_a$'s of the clusters differ significantly,
then the leading eigenvalues of $\B$ are not closely aligned with those of
$\RR \C$. However, this case is relatively simple, as then the clusters could
be formed by sorting the node degrees.

\section{Inflation--deflation}\label{id}

Here we use the inflation--deflation technique 
of~\cite{Bolla13}. %(p. 100).

With the notation of the $SBM_k$ model: $\C =(c_{ab})$ is the $k\times k$
symmetric affinity matrix (assume that its rank is $k$)
and $\PP =\frac1{n} \C$ is the 
corresponding probability matrix of the random graph $G_n \in SBM_k$.
If self-loops are allowed, then $\bar \A$ is the expected adjacency matrix of
$G_n$ and it contains $p_{ab}$'s in the $(a,b)$-th block. Therefore,
it is the inflated matrix of $\PP$ with blow-up sizes $n_1 ,\dots ,n_k$
$(\sum_{a=1}^k n_a =n )$. The blow-up ratios are $r_a =\frac{n_a}{n}$ for
$a=1,\dots ,k$. The expected average degree of $G_n$ is
$$
c= \frac1{n} \sum_{i=1}^n \sum_{i=1}^n {\bar a}_{ij} =
%\sum_{a=1}^k \sum_{b=1}^k r_a r_b c_{ab} =
\frac1{n} \sum_{a=1}^k \sum_{b=1}^k n_a n_b c_{ab}
 =\sum_{a=1}^k r_a  c_a ,
$$
as in~\eqref{deg}.
Note that $c$ depends only on the model parameters and it is of constant order.
If the diagonal  entries of $\A$ are zeros, then $\bar \A$
differs from the expected adjacency matrix by
$\diag (p_{11}, \dots ,p_{11} , \dots , p_{kk} ,\dots ,p_{kk} )$ with $n_a$
entries equal to each other in the $a$-th block. 
However, by the Weyl's perturbation theorem,
the $k$ leading eigenvalues of $\bar \A$ and $\E (\A)$ differ only with a
term of order $\frac1{n}$. 

\begin{proposition}\label{step}
The matrix ${\bar \A}$ has rank $k$ and its non-zero eigenvalues ($\nu$'s)
with unit norm eigenvectors ($\uuu$'s) satisfy
\begin{equation}\label{u}
 {\bar \A } \uuu = \nu \uuu ,
\end{equation}
where $\uuu$ is the inflated vector (step-vector) of 
${\tilde \uuu } = (u (1) ,\dots ,u(k ))^*$ with block-sizes $n_1 ,\dots ,n_k$.
With the notation $\RR =\frac1{n} \diag (n_1 ,\dots ,n_k ) =
\diag (r_1 ,\dots ,r_k )$, the deflated equation of~\eqref{u} is equivalent to
\begin{equation}\label{v}
 \RR^{\frac12} \C \RR^{\frac12} \vvv = \nu \vvv ,
\end{equation}
where $\vvv =\sqrt{n}\RR^{\frac12} {\tilde \uuu}$.
Further,
if $\uuu_1 ,\dots \uuu_k$ is the set of orthonormal eigenvectors, corresponding
to the eigenvalues $\nu_1 ,\dots ,\nu_k$ of
$\bar \A$, then
$\vvv_i =\sqrt{n}\RR^{\frac12} {\tilde \uuu}_i$ $(i=1,\dots ,k )$ is the set
of orthonormal eigenvectors of $\RR^{\frac12} \C \RR^{\frac12}$
with the same eigenvalues.
Also, $\RR^{\frac12} \vvv_i = \sqrt{n}\RR {\tilde \uuu}_i$ is a right 
eigenvector of
$\RR \C$ and $\RR^{-\frac12} \vvv_i = \sqrt{n} {\tilde \uuu}_i$ is a left
eigenvector of $\RR \C$ 
with eigenvalue $\nu_i$, for $i=1,\dots ,k$.
\end{proposition}

\begin{proof}
The matrix ${\bar \A}$ has $k$ different rows (columns), so it is of rank $k$ 
and its non-zero eigenvalues ($\nu$'s) have corresponding
unit norm eigenvectors ($\uuu$'s) in~\eqref{u}, which  are trivially
step-vectors, i.e., $\uuu$ is the inflated vector of, say,
${\tilde \uuu } = (u (1) ,\dots ,u(k ))^*$ with block-sizes $n_1 ,\dots ,n_k$.
The deflated equation of~\eqref{u} is equivalent to
$$
 \RR^{\frac12} \PP \RR^{\frac12} \vvv = \frac{\nu}{n} \vvv ,
$$
that in turn is equivalent to~\eqref{v},
where the eigenvector $\vvv =\sqrt{n}\RR^{\frac12} {\tilde \uuu}$ also has unit 
norm.

If $\uuu_1 ,\dots \uuu_k$ is the set of orthonormal eigenvectors of 
$\bar \A$, then  $\vvv_i =\sqrt{n}\RR^{\frac12} {\tilde \uuu}_i$ 
$(i=1,\dots ,k )$  
is the set of orthonormal eigenvectors of $\RR^{\frac12} \C \RR^{\frac12}$
with the same eigenvalues.
Also, 
$$
 \RR \C (\RR^{\frac12} \vvv_i ) = 
 \RR^{\frac12} ( \RR^{\frac12} \C \RR^{\frac12} ) \vvv_i  =
\RR^{\frac12} \nu_i\vvv_i =\nu_i\RR^{\frac12} \vvv_i ,
$$
so $\RR^{\frac12} \vvv_i = \sqrt{n}\RR {\tilde \uuu}_i$ is a right eigenvector 
of $\RR \C$ for $i=1,\dots ,k$. Further,  
$$
 (\RR^{-\frac12} \vvv_i)^* \RR \C = 
  \vvv_i^* (\RR^{\frac12} \C \RR^{\frac12}) \RR^{-\frac12} =
 \nu_i \vvv_i^* \RR^{-\frac12} = \nu_i (\RR^{-\frac12} \vvv_i)^* ,
$$
so $\RR^{-\frac12} \vvv_i = \sqrt{n} {\tilde \uuu}_i$ is a left
eigenvector of $\RR \C$ 
with the same eigenvalue $\nu_i$, for $i=1,\dots ,k$.

Observe that $(\RR \C )^* =\C \RR$, so the non-symmetric matrices $\RR \C$
and $\C \RR$ have the same set of eigenvalues, which are also real, as
they are identical to the eigenvalues of the symmetric matrix
$\RR^{\frac12} \C \RR^{\frac12}$. Further, right eigenvectors of $\RR \C$
are left eigenvectors of $\C \RR$, and vice versa.
%left eigenvectors of $\RR \C$ are right eigenvectors of $\C \RR$.
This finishes the proof.
\end{proof}

\begin{remark}\label{rem1} We remark the following.
  
  \begin{itemize}
  \item Let the real eigenvalues of $\bar \A$ be ordered in non-increasing
    order. Since $\bar \A$ is an irreducible matrix of positive entries,
    by the Frobenius theorem,
    $\nu_1$ is a single positive eigenvalue and $|\nu_i | <\nu_1$ for
    $i=2,\dots ,k$. Further, the eigenvector $\uuu_1$ or $\bar \A$,
    corresponding to $\nu_1$, has all positive coordinates.

   \item 
     In particular, if $c_1 =\dots =c_k =c$, then $\frac1{c} \C \RR$ is a
     stochastic matrix with largest eigenvalue 1 and corresponding
     right eigenvector $\1 \in \R^k$. Consequently, the largest eigenvalue
     of $\frac1{c} \RR \C$ is also 1 with
     left eigenvector $\1 \in \R^k$. Therefore, $\nu_1 =c$ and
     ${\tilde \uuu_1} =\frac1{\sqrt{n}} \1$ also has equal coordinates.
     So $\uuu_1$ is such too, and
     $\uuu_1 =\frac1{\sqrt{n}} \1 \in \R^n$ is the unit-norm eigenvector of
     $\bar \A$, corresponding to $\nu_1 =c$ in this special case.

\item
     If we further assume
the $c_{in}$ versus $c_{out}$ scenario, then in the symmetric case,
$\nu_2 =\dots =\nu_k  =c\lambda$ with $0<|\lambda | <1$ 
(see Section~\ref{betaperc}).
\item
Note that under the condition~\eqref{ra} of Section~\ref{bp},
the eigenvalues of $\RR \C$ and $\bar \A$ may differ within a factor
approaching 1.
\item
  In the unpercolated situation, the edge-probabilities are kept constant,
  $\PP =\C$ and the $n\times n$ $\bar \A$ is the blown-up
  matrix of the $k\times k$ probability matrix $\PP$. So, its
  eigenvalues are $n$ times the eigenvalues of $\RR^{-1/2} \C \RR^{-1/2}$.
  The Wigner-noise added has spectral norm $O (\sqrt{n})$(see~\cite{Bolla13}),
  therefore it cannot suppress the structure, unlike in the sparse
  $SBM_k$ model, discussed in the following.
\end{itemize}
\end{remark}

The (random) adjacency matrix $\A$ of (the random graph) $G_n$
coming from the sparse $SBM_k$ model is  $\A ={\bar \A } +\W$, where
$\W$ is an appropriate (random) error matrix and all the matrices are 
$n\times n$
symmetric (though we do not use an index $n$ for this fact).
We can achieve that the matrix $\A$ contains 1's in the $(a,b)$-th
block with probability $p_{ab}$, and 0's otherwise. Indeed, for indices
$1\le a<b \le k$ and $i\in V_a$, $j\in V_b$  let
\begin{equation}\label{wig1} 
w_{ji} =w_{ij} := \left\{ \begin{array}{ll}
   1-p_{ab} & \mbox{with probability } \quad p_{ab}  \\
   -p_{ab}   & \mbox{with probability } \quad 1-p_{ab} 
    \end{array}
 \right.
\end{equation}
be independent random variables; further, for
$a=1,\dots ,k$ and $i,j\in V_a$  $(i\le j)$  let
\begin{equation}\label{wig11}  
w_{ji} =w_{ij} := \left\{ \begin{array}{ll}
   1-p_{aa} & \mbox{with probability } \quad p_{aa}  \\
   -p_{aa}   & \mbox{with probability } \quad 1-p_{aa} 
    \end{array}
 \right. 
\end{equation}
be also independent, otherwise $\W$ is symmetric.
If self-loops are not allowed, then $w_{ii} =-p_{aa}$ with probability 1
could be defined for $i\in V_a$. It won't change significantly the
subsequent calculations.

This $\W$ is not a Wigner noise (see~\cite{Bolla13}) as it does not have a
nested structure.  However, if $i\in V_a$ and  $j\in V_b$, then 
%its entries are uniformly bounded with zero
%expectation and variance
%$$
%  p_{ab} (1-p_{ab}) = \frac{c_{ab} (n-c_{ab})}{n^2},\quad i\in V_a , j\in V_b .% $$
%  Actually, here ${\sqrt{n}} \W$ is a traditional Wigner noise, as
 $$
 \Var (\sqrt{n} w_{ij} )=  n p_{ab} (1-p_{ab})=n \frac{c_{ab} (n-c_{ab})}{n^2} =
   c_{ab} (1-\frac{c_{ab}}{n} ) \le c_{ab} ,
   $$
   where, without hurting the generality, we may assume $0<c_{ab} <1$.
   (Even if for some $a,b$, $c_{ab} \ge 1$, with large enough $n$,
   $\frac{c_{ab}}{n} <1$ holds.)
 Therefore,  $\sqrt{n}\W$ is a general Wigner-type matrix with
 standard deviation of its entries bounded by $\sigma = \max_{a,b}\sqrt{c_{ab}}$.
 With constant $\sigma$ for each entries of
 $\sqrt{n}\W$, in~\cite{Capitaine} it is proved that the spectrum of
 $\W$ obeys the semicircle law with radius $2\sigma$.
 More generally, the authors of~\cite{Khorunzhy} establish similar result for
 a Wigner-type matrix with independent (in and above the main diagonal),
 but usually not identically distributed entries; also, the authors
 of~\cite{Achlioptas} have
 analogous findings for the spectral norms of rectangular matrices.
 Latter results only require the
 fourth moments of the entries to exist. (Note that sometimes all the
 moments are required to exist, or the Poincar\'e inequality is required
 to be satisfied.)
 In~\cite{Sohsnikov}, universality for the leading eigenvalues of a
 Wigner-type matrix is established.
 
 Now, the finite rank matrix $\bar \A$ is
 considered as a perturbation on $\W$.
By the theory of deformed Wigner matrices~\cite{Capitaine}, the relation
 between $\sigma$ and the non-zero eigenvalues of $\bar \A$ decides the
 situation. Since the largest eigenvalue of the matrix $\bar \A$ of rank
 $k$ is between $\min_a c_a$ and $\max_a c_a$, it is typically
 smaller than $\sigma = \max_{a,b} \sqrt{c_{ab} }$,
 therefore the largest eigenvalue of $\bar \A +\W$
 is within the semicircle, and the structure in $\bar \A$ is suppressed by
 the noise $\W$.
 By \cite{Capitaine}, a non-zero eigenvalue $\nu_j$ of $\bar \A$ is
 perturbed as $\nu_j +\frac{\sigma^2}{\nu_j}$, if $\nu_j >\sigma$.
 The deviation from $\nu_j$ is of constant order, it does not decrease
 with $n$. So it is either swallowed by the noise or it is not
 changed with $n$.
Therefore, the relation of the non-zero eigenvalues of
 $\bar \A$ and  $\W$ depends of the individual $c_{ab}$s, but
 in general, it is impossible to distinguish between them.

Luckily, this is not the case with the eigenvalues of $\B$. 
%Summarizing, the eigenvalues of the inflated matrix $\bar \A$
%(same as those of $\RR \C$)  are of constant order, but 
%by the Weyl's perturbation theorem, a constant order
%of perturbation ($\W $) can take them farther from those of the 
%true adjacency matrix ($\A$) with distance $\max_a (2c_a +\sqrt{c_a})$
%than from the leading eigenvalues of $\B$.
%The same is true for the relation between the $k$ leading
%eigenvalues of the expected adjacency matrix $\E (\A )$ and $\A$.
As we saw, $\B$'s leading eigenvalues are within a factor 
$\max_a \frac{c}{c_a}$ from those of $\RR \C$.
In this way, the additive errors can be larger than the multiplicative ones.
%This is also supported by the theory of deformed Wigner
%matrices~\cite{Capitaine}.

As for the eigenvectors, in the $SBM_k$ model, the eigenvectors of the
expected adjacency matrix $\bar \A$
are $n$-dimensional step-vectors having $k$ different
coordinates. If $\uuu$ denotes such an eigenvector, then
the corresponding $2m$-dimensional 
inflated vector $\bm{End} \, \uuu$ has the same coordinates with
multiplicities equal to the node-degrees $(k\le n\le m)$. Those are also 
step-vectors on $k$ steps, and by Proposition~\ref{propo}, are close to
the corresponding eigenvector of $\B$ (with the close eigenvalue).
Therefore, the deflated vector is
$$
 {\bm{End}}^* {\bm{End}} \, \uuu =  ({\bm{End}} \, \uuu )^{in}
$$
and, by Equation~\eqref{es}, 
$\DD^{-1} \, ({\bm{End}} \,  \uuu )^{in}$ approximates $\uuu$;
i.e., the first $n$ coordinates of the corresponding $\K$-eigenvector,
divided with the corresponding node-degrees, will
 approximate $\uuu$. 

 Recall Proposition~\ref{propo}, according to which the
 leading positive real eigenvalues
of $\B$ (those of $\K$) are close to those of the expected adjacency matrix.
In case of the conditions of Proposition~\ref{propo},
the corresponding normalized
eigenvectors of $\B$ have an inner product with the
standardized $\bm{End}$-transforms of the analogous eigenvectors
of the expected adjacency matrix very close to 1. But the eigenvectors
of the latter one are step-vectors, and the eigenvectors of the former
have a large inner product with them, so they are within a ``small''  distance
from the subspace spanned by the two sets of the $k$ leading eigenvectors.
This implies that without knowing the model parameters, the $k$-variance
of the $k$-dimensional representatives of the directed edges is ``small''.
In this way, we obtain $k$
clusters of the bidirected edges, but via the majority of their
end-nodes we can classify the majority of the nodes too.

To decrease the computational complexity, we can as well use the
right eigenvectors  of $\K$ (only one half of them), more precisely,
the vectors $\DD^{-1} \, \x_i^{in}$ for $i=1,\dots ,k$.
For the representatives, based on them,
the k-means algorithm is applicable.

More precisely, by~\cite{Stephan}, let $\x$ be a unit-norm eigenvector of $\B$, 
corresponding to the eigenvalue $\mu$ that is close to the eigenvalue $\nu$
of the expected adjacency matrix, with corresponding eigenvector 
$\uuu \in \R^n$.
If our graph is from the $SBM_k$ model, then (without knowing its parameters)
we know that $\uuu$ is a step-vector with at most $k$ different coordinates.  
In~\cite{Stephan} it is proved that 
$$
 \left \langle \x , \frac{{\bm{End}} \, \uuu }{\| {\bm{End}} \, \uuu \| } 
\right \rangle \ge \sqrt{1-\ep} \ge 
 1-\frac12 \ep ,
$$
where $\ep$ can be arbitrarily ``small'' with increasing $n$.   Consequently,
$$
\left \|  \x - \frac{{\bm{End}} \, \uuu }{\| {\bm{End} } \, \uuu \| } 
\right \|^2 \le  2 - 2(1-\frac12 \ep ) = \ep   
$$
and, as $\x^{in} = \bm{End}^* \x$ holds by~\eqref{inout} and
${\bm{End}}^*{\bm{End}} =\DD $ holds by~\eqref{es},
$$ 
\left \| {\bm{End}}^* \x - {\bm{End}}^* \frac{{\bm{End}} \, \uuu }{\|
    {\bm{End} }  \, \uuu \|} \right \|^2  =
 \left \| \x^{in} - \DD \frac{\uuu }{\| {\bm{End} } \, \uuu \|} \right \|^2 
$$
also holds.
Consequently,
$$ 
 \left \| \DD^{-1} \x^{in} - \frac{\uuu }{\| {\bm{End} } \, 
\uuu \|} \right \|^2 \le \| \DD^{-1} {\bm{End}}^* \|^2 \ep \le \ep.
$$
Indeed, the largest eigenvalue of  
$(\DD^{-1} {\bm{End}}^* ) ({\bm{End}}\,\DD^{-1} ) = \DD^{-1}\DD \DD^{-1} =\DD^{-1}$
is $\max_i \frac1{d_i}$, so the largest singular value (spectral norm) of
$\DD^{-1} {\bm{End}}^*$ is estimated from above with 
$\left( \max_i \frac1{d_i} \right)^{\frac12}$.
Therefore,
$$
 \| \DD^{-1} {\bm{End}}^* \|^2 \le \max_i \frac1{d_i} = \frac1{\min_i d_i} \le 1.
$$

%Now we apply this to the $k$ leading normalized eigenvectors 
%$\x_1 ,\dots ,\x_k$
%of $\B$, for which
%$$
% \sum_{j=1}^k \| \DD^{-1} \x_j^{in} - \frac {\uuu_j }{\| 
%      {\bm End } \, \uuu_j \|} \|^2 \le k \ep \max_i d_i .
%$$
%As $\uuu_j$s are step-vectors with $k$ different coordinates on the same
%$k$ steps, the above sum of the squares is the objective function of the
%k-means  algorithm. Without knowing the $\uuu_j$s, we minimize it with
%the $k$-dimensional node representatives 

This results can be summarized in the following theorem.
\begin{theorem}\label{kvar}
Assume that the expected adjacency  matrix of the underlying random graph on
$n$ nodes and $m$ edges has rank $k$ with $k$ single eigenvalues  
and corresponding unit-norm eigenvectors 
$\uuu_1 ,\dots ,\uuu_k \in \R^n$. 
Assume that the non-backtracking matrix $\B$ of
the random graph has $k$ structural eigenvalues (aligned with those of the
expected adjacency matrix) with eigenvectors
$\x_1 ,\dots ,\x_k \in \R^{2m}$ such that
$$
 \left \langle \x_j , \frac{{\bm{End}} \, \uuu_j }{\| {\bm{End} } \, \uuu_j \|} 
\right \rangle \ge \sqrt{1-\ep} , \quad j=1,\dots ,k.
$$
 Then
for the transformed vectors $\DD^{-1} \x_j^{in} \in \R^n$, the relation
\begin{equation}\label{becs}
 \sum_{j=1}^k  \left\| \DD^{-1} \x_j^{in} - \frac {\uuu_j }{\| 
      {\bm{End} } \, \uuu_j \|} \right\|^2 \le k \ep  
\end{equation}
holds, where $\DD$ is the degree-matrix of the underlying graph.
\end{theorem}

\begin{corollary}
If $\uuu_j$'s are step-vectors on $k$ steps (e.g., if our graph comes from the
$SBM_k$ model), then 
the $k$-variance of the node representatives 
$$
 ( \frac1{d_i} x_{1i}^{in} , \dots , \frac1{d_i} x_{ki}^{in} ),
 \quad i=1,\dots ,n 
$$
is estimated from above with $k \ep$ too. Indeed, the $k$-variance
is the minimum squared distance between the subspace spanned by the vectors 
$\DD^{-1} \x_j^{in}$ $(j=1,\dots ,k)$ and that of the $k$-step vectors, 
see~\cite{Bolla11} for the proof. Since in the $SBM_k$ model, $\uuu_j$'s are
also step-vectors on the same $k$ steps (as was proved in 
Proposition~\ref{step}),  this $k$-variance is less than the left-hand side
of inequality~\eqref{becs}, so it is at most  $k \ep$. 

Also, if the underlying graph is sparse and the conditions of Theorem 1 
of~\cite{Stephan} hold, then
$\ep >0$ can be arbitrarily small with increasing $n$. 
To minimize the objective function of the k-means algorithm, 
there are polynomial time approximating
schemes, e.g.,~\cite{Ostrovsky}, so without knowing the model parameters and the
vectors $\uuu_j$s, we can find the clusters with the help of the leading
$\B$-eigenvectors only.

It is important that here $k$ is the number of the structural eigenvalues of
$\B$, and the expected adjacency matrix must have a ``good'' $k$-rank
approximation.

This gives rise to a future research on sparse spectral clustering.
\end{corollary}

\begin{remark}
In case of an unweighted graph, the $n$-dimensional vector 
$\x_j^{in} $, obtained from a right eigenvector of $\B^*$, 
is the second segment of the right eigenvector of the matrix $\K$,
corresponding to the same eigenvalue. Consequently, by the swapping issues of
Section~\ref{pre}, the $n$-dimensional vector 
$\x_j^{in} $,  obtained from a right eigenvector of $\B$, 
is the first segment of the right eigenvector of $\K$,
for $j=1,\dots ,k$. 
So, we have to perform the spectral decomposition of a $2n\times 2n$
matrix only instead of a $2m\times 2m$ one, which fact has further computational
benefit (except for trees, $n\le m$, but usually $n$ is much smaller than $m$).

Also note that, by Remark~\ref{rem1}, in the $c_1 =\dots =c_k =c$
special case, $\uuu_1 =\frac1{\sqrt{n}} \1 \in \R^n$, and therefore, the
normalized $ {\bm{End} } \uuu_1 \in \R^k$ also has equal coordinates. 
Consequently, the vectors $\x_1$ and $\DD^{-1} \x_1^{in}$ also nearly have
the same coordinates, so the first coordinates of the $k$-dimensional
node representatives can be disregarded, and instead we use
$(k-1)$-dimensional representatives.

\end{remark}

\section{EM algorithm for estimating the parameters of the $SBM_k$ model}\label{em}

Given a simple graph
$G=(V, \A )$ ($|V|=n$)  with adjacency matrix $\A$ and an integer 
$k$ $(1<k<n)$, we are looking for the hidden $k$-partition $(V_1 ,\dots ,V_k )$
of the nodes such that
\begin{itemize}
\item nodes are independently assigned to cluster $V_a$  with
probability $r_a$, $a=1,\dots ,k$; $\sum_{a=1}^k r_a =1$;
\item given the cluster memberships, nodes of $V_a$ and $V_b$ are connected
independently, with probability 
$$
 \PPP (i\sim j \, | \, i\in V_a , j\in V_b ) = p_{ab}, \quad 1\le a,b\le k.
$$
\end{itemize}
The parameters are collected in the probability vector 
$\rrr = (r_1 ,\dots ,r_k)$
and the $k\times k$ symmetric probability matrix $\PP =(p_{ab})$.
We could as well consider the affinity matrix, but it does not make much
difference as $n$ is fixed now.

Here the $n\times n$ symmetric, 0-1 adjacency  matrix $\A =(a_{ij})$ 
of $G$ represents a statistical sample.
There are no loops, so the diagonal entries are
zeros, whereas the independent off-diagonal entries are considered as an i.i.d.
sample.  Based on $\A$, we want to estimate the parameters
$\rrr$ and $\PP$ of the above  $SBM_k$ model.
In fact, $\A$ is the incomplete data specification
as the cluster memberships are missing. Therefore, 
it is straightforward to use the
EM algorithm
%proposed by~\citet{DLR} %, also discussed e.g. in~\citet{Hastie,McLachlan}
for parameter estimation from incomplete data.
%This special application for mixtures is sometimes called 
%\textit{collaborative filtering}. % see~\citet{HP,Ungar}.

First we complete our data matrix
$\A$ with latent membership vectors $\Dada_1 ,\dots ,\Dada_n$ of the nodes
that are $k$-dimensional i.i.d. $Poly (1,\rrr )$ (multinomially distributed)
random vectors. More precisely,
$\Dada_i =(\Delta_{1i},\dots ,\Delta_{ki} )$, where $\Delta_{ai} =1$
if $i\in V_a$ and zero otherwise. 
Thus, the sum of the coordinates of any $\Dada_i$ is 1,
and $\PPP (\Delta_{ai} =1 ) =r_a$.

The likelihood is the square-root of
\begin{equation}\label{likmod}
 \prod_{a=1}^k \prod_{i=1}^n \prod_{b=1}^k
 [ p_{ab}^{\sum_{j:\, j\ne i} \Delta_{bj} a_{ij}}   \cdot
 (1-p_{ab})^{ \sum_{j:\, j\ne i} \Delta_{bj} (1-a_{ij})} ]^{\Delta_{ai}} ,
\end{equation}
which is maximized by the following iteration.

Starting with initial parameter values $\rrr^{(0)}$, $\PP^{(0)}$ and
membership vectors $\Dada_1^{(0)} ,\dots ,\Dada_n^{(0)}$, the $t$-th
step of the iteration is the following ($t=1,2,\dots $).
\begin{itemize}
\item \textbf{E}-step: 
we calculate the conditional expectation of each $\Delta_i$ conditioned
on the model parameters and on the other cluster assignments obtained in step
$t-1$ and collectively denoted by $M^{(t-1)}$.
By the Bayes rule, the responsibility of node $i$ for cluster $a$ is
$$
\begin{aligned}
 r_{ai}^{(t)} &=\E (\Delta_{ai} \, | \, M^{(t-1)}) \\
 &=\frac{\PPP (M^{(t-1)} |\Delta_{ai}=1 ) \cdot r_a^{(t-1)} }
{\sum_{b=1}^k \PPP (M^{(t-1)} |\Delta_{bi}=1 ) \cdot r_b^{(t-1)} }
\end{aligned}
$$($a=1,\dots ,k; \, i=1,\dots ,n$).  
For each $i$, $r_{ai}^{(t)}$ is proportional to the numerator, 
where
$$
\begin{aligned}
 &\PPP (M^{(t-1)} |\Delta_{ai}=1 )  \\
 = &\prod_{b=1}^k 
 (p_{ab}^{(t-1)})^{\sum_{j\ne i} \Delta_{bj}^{(t-1)} a_{ij}}  \cdot
 (1-p_{ab}^{(t-1)} )^{ \sum_{j\ne i}  \Delta_{bj}^{(t-1)} (1-a_{ij})}
\end{aligned}
$$
is the part of the likelihood~(\ref{likmod}) that affects node $i$ under
the condition $\Delta_{ai}=1$.
\item \textbf{M}-step:  we maximize the truncated binomial likelihood 
$$
 p_{ab}^{\sum_{i\ne j} r_{ai}^{(t)} r_{bj}^{(t)} a_{ij}}  \cdot
 (1-p_{ab})^{ \sum_{i\ne j} r_{ai}^{(t)} r_{bj}^{(t)} (1-a_{ij})}
$$
with respect to the parameter $p_{ab}$, for all $a,b$ pairs separately. 
Obviously, the maximum is attained
when the following estimators of $p_{ab}$'s comprise 
the symmetric matrix $\PP^{(t)}$:
$
p_{ab}^{(t)}=\frac{\sum_{i,j:\, i\ne j} r_{ai}^{(t)} r_{bj}^{(t)} a_{ij}}
{\sum_{i,j:\, i\ne j}r_{ai}^{(t)} r_{bj}^{(t)} }$
$(1\le a\le b\le k)$,
where edges connecting nodes of clusters $a$ and $b$ are counted 
fractionally, multiplied by the membership probabilities of their endpoints.
\end{itemize}

The maximum likelihood estimator of $\rrr$ in the $t$-th step is $\rrr^{(t)}$ of
coordinates  $r_a^{(t)} =\frac1{n} \sum_{i=1}^n 
 r_{ai}^{(t)}$ ($a=1,\dots ,k$), while that of the membership vector
$\Dada_i$ is obtained by discrete maximization: $\Delta_{ai}^{(t)} =1$
if $r_{ai}^{(t)} = \max_{b\in \{ 1,\dots ,k \} } r_{bi}^{(t)}$ and
0, otherwise. 
(In case of ties, nodes are classified arbitrarily.) 
This choice of $\rrr$ will increase (better to say, not decrease)
the likelihood function.
Note that it is not necessary to assign nodes uniquely to the clusters,
the responsibility $r_{ai}$ of a node $i$ can as well be regarded as
the intensity of node $i$ belonging to cluster $a$,
akin to the fuzzy clustering.

%The above algorithm is a special case of so-called Collaborative 
%Filtering~\cite{Hof}.
According to the general theory of the EM algorithm~\cite{DLR},
in exponential families (as in the present case), convergence
to a local maximum can be guaranteed 
(depending on the starting values), but it runs in polynomial time in 
the number of nodes $n$. However, the speed and limit of the convergence 
depend on the starting clustering, which can be chosen by means of
preliminary application of some nonparametric multiway cut algorithm.
% of Section~\ref{nonpara}.

\section{Special cases and $\beta$-percolation}\label{betaperc}

In the most special ``symmetric case'', the transition matrix is 
$\T=\G \RR \C =\frac 1{ck} \C$, where $\C$  contains $c_{in}$ in its 
main diagonal and $c_{out}$, otherwise;
see the BP method of Section~\ref{bp}.
%The belief propagation, called message
%passing or cavity method, is treated in in~\cite{Moore}.
$\T$ is a stochastic matrix, so its largest eigenvalue is 1 with corresponding
eigenvector $\1$ (the all 1's vector). The other eigenvalue is
\begin{equation}\label{lam}
 \lambda = \frac{c_{in} -c_{out} }{kc} 
\end{equation}
with multiplicity $k-1$. 
In the assortative case, $\lambda >0$; further, $\lambda <1$, as in the 
symmetric case
$$
 c =\frac{c_{in} + (k-1) c_{out}}{k}
$$
holds.
Consequently, the eigenvalues of $\RR \C =c \T$ are $c$  and $c\lambda$,
latter one has  multiplicity $k-1$.

With the BP method of~\cite{Moore}, that treats this special case,
the eigenvalues
of $\B \otimes \T$ greater than 1 are considered
%(giving a non-trivial solution)
that boils down to the condition $c\lambda^2 >1$, which gives the  
Kesten--Stigum threshold~\eqref{2}.

Now assume that $G_n \in SBM_k^{\beta}$. Then the $k\times k$ 
probability matrix is $(p_{ab}) =\beta \frac1{n}\C =\frac{\beta \C}{n}$.
Consequently, $\C$ and $c$ is multiplied by $\beta$, 
but $\T =\G \RR \C$ remains  unchanged.
So we consider  $\beta c\T$ as for the model side, 
but the underlying graph
and its $\B$ is the same as before. Therefore, the eigenvalues of
$\B \otimes \beta c \T = \beta c (\B \otimes \T )$ should be greater than 
$c$ if a non-trivial solution is required:
$$
 \beta \, \lambda (\B \otimes c \T ) > c .
$$
If the eigenvalues of $\B$ and $c\T$ are aligned, then this gives that
$
 \lambda (\B) > \frac{\sqrt{c}}{\sqrt{\beta}}
$
is needed for detectability; equivalently,
\begin{equation}\label{bet}
 \beta > \frac{c}{\lambda^2 (B)} =\left(\frac{\sqrt{c}}{\lambda (\B )}\right)^2.
\end{equation}

This is in accord with the fact, that in the $k=1$ case, in the
Erd\H os--R\'enyi model, when the largest  eigenvalue
of $\B$ is around $c$, then $\beta = %\frac1{\lambda (B)} =
\frac1{\mu_1}$ is  the percolation threshold, see~\cite{Newman22}.
In the multiclass scenario, $\frac{c}{\mu_i^2}$
are further phase transitions, leading to $i$ clusters, for $i=1,\dots ,k_0$
until $\mu_{k_0} \ge \sqrt{c}$, but $\mu_{k_0+1} <\sqrt{c}$.
Also note that inequality~\eqref{bet} has relevance only if $\lambda (\B) >
\sqrt{c}$, so eigenvalues of $\B$ greater than $\sqrt{c}$ give the phase 
transitions.

It is also in accord with the findings in the forthcoming $k=2$ case.
Let us consider the symmetric $SBM_2$ model: $r_1 =r_2 =\frac12$, $c_1 =c_2 = 
\frac{c_{in} + c_{out}}2 =c$. Then condition~\eqref{2} 
in the assortative case boils down to
\begin{equation}\label{22}
  c_{in} - c_{out}=|c_{in} -c_{out} |>2\sqrt{c} =\sqrt{2}\sqrt{c_{in} + c_{out}} .
\end{equation}
If $\C$ and $c$ are multiplied with $\beta$, we get
$
 \beta ( c_{in} - c_{out} )>\sqrt{2} \sqrt{\beta (c_{in} + c_{out}) } 
$.
This means that
$$
  c_{in} - c_{out} >\sqrt{2} \frac{\sqrt{c_{in} + c_{out}}}{\sqrt{\beta}} .
$$
Since $\beta <1$, the right hand side gives a higher lower threshold than 
$\beta =1$ in Equation~\eqref{22}.
This is also equivalent to
\begin{equation}\label{2beta}
  \beta >\frac{2(c_{in} + c_{out})}{(c_{in} - c_{out})^2} ,
\end{equation}
which makes sense if $\frac{2(c_{in} + c_{out})}{(c_{in} - c_{out})^2} <1$,
i.e., if
$$
  c_{in} - c_{out} >2\sqrt{c} =\sqrt{2} \sqrt{c_{in} + c_{out}} ,
$$
in accord with~\eqref{22}.
So, until equality is attained in~\eqref{2beta}, additional 
$\beta$-percolation can give detectable clusters too.

In the special case when $c_1 =\dots =c_k =c$, the above is also equivalent to
$$
c \lambda >\sqrt{c} ,
$$
where $c\lambda$ is the second largest eigenvalue (of multiplicity $k-1$)
of $c \T =\RR \C$, and it is closely related to the eigenvalues
$\mu_2 ,\dots ,\mu_k$ of $\B$. 

The above is also in accord with~\eqref{bet}, since the threshold for 
$\beta$ is
$$
 \beta> \frac{c}{\mu_2^2} = \frac{c}{ (\frac{c_{in}-c_{out}}2 )^2 } =
 \frac{ \frac{c_{in} +c_{out}}2}{ \frac{(c_{in}-c_{out})^2 }4  } =
 \frac{2(c_{in} + c_{out})}{(c_{in} - c_{out})^2} ,
$$
which is the same as inequality~\eqref{2beta}.

Observe that inequality~\eqref{2} implies
$$
 \beta | c_{in} - c_{out} | >k\sqrt{\beta c} ,
$$
which yields 
$$
 \beta > \frac{k^2 c}{( c_{in} - c_{out} )^2} = \frac{ k [c_{in}+(k-1) c_{out}]}
 {( c_{in} - c_{out} )^2} =\frac{c}{(c \lambda)^2} =
 \left( \frac{\sqrt{c}}{\lambda} \right)^2 
$$
in the symmetric case, where $\lambda$ is the second largest eigenvalue of
$\RR \C$ with multiplicity $k-1$ and it is aligned with $\mu_2 ,\dots ,\mu_k$.
So in the symmetric case, as the second largest eigenvalue of $\RR \C$
has multiplicity $k-1$, it gives rise to $k-1$ very close phase transitions.
In other cases, those can be farther apart.

Another implication is that
$$
 \sqrt{\beta} | c_{in} - c_{out} | >k\sqrt{ c} ,
$$
where the left-hand side decreases with $\beta <1$. It can be decreased until
equality is attained, when $\beta =\frac{c}{\lambda^2}$. 

In the general scenario, it means that at $\frac{c}{\mu_i^2}$, 
for $i=2,\dots ,k$,
there are phase transitions for $\beta$ producing $2,\dots ,k$ detectable
clusters: 
$$
 | c_{in} - c_{out} | >k\frac{\sqrt{c}}{\sqrt{\beta}} > k\sqrt{c} ,
$$
so there is a leverage of the percolation threshold with $1/\sqrt{\beta}$.

In the $c_{in}$ versus $c_{out}$ scenario,  $c_1 =\dots =c_k =c$ is equivalent to
 $r_1 =\dots =r_k$, and so,
$$
 \beta | c_{in} - c_{out} | >k\sqrt{\beta c} , \quad 
  | c_{in} - c_{out} | >k\frac{\sqrt{ c}}{\sqrt{\beta}} 
  $$
  which allows more detectable clusters if $\beta$ is increased,  see the
  subsequent illustrations.

Eventually, this phenomenon is illustrated by simulations.
First a random graph was generated on $n=900$ nodes, with parameter matrices
$\RR =\diag (\frac3{10},\frac13,\frac{11}{30})$ and 
$$
 \C =\begin{pmatrix}
 30  & 12 & 10 \\
 12  & 32 & 9 \\
 10 & 9 & 27
\end{pmatrix} .
$$

Then we percolated (retained) the edges with probability $\beta$. 
For small $\beta$'s the graph was not connected, which fact caused the
eigenvalue $1$ of the matrix $\K$ to have several geometric, and even more
algebraic multiplicities.
With increasing $\beta$, 
 Figures~\ref{fig:2},~\ref{fig:5},~\ref{fig:7}, and~\ref{fig:8} show the
 adjacency matrices and the $\K$-spectrum of the $\beta$-percolated
 random graph, when 1,2, and 3 structural (positive real) eigenvalues emerge
out of the bulk (within the encircled region). This happens about when $\beta$
exceeds $\beta_1$, $\beta_2$, and $\beta_3$, but the phase transitions are not
strict, as the alignment of the structural eigenvalues of the non-backtracking 
matrix $\B$ (those of $\K$)
and those of the model matrix $\RR \C$ is just asymptotic (for ``large'' $n$)
and guaranteed only when the expected degrees of the clusters are 
approximately the same.

\begin{figure}[htbp]
\centering  
\subfigure[Adjacency matrix (black: 1, white: 0).]{  
\begin{minipage}{5.6cm}
\centering  
\includegraphics[scale=0.4]{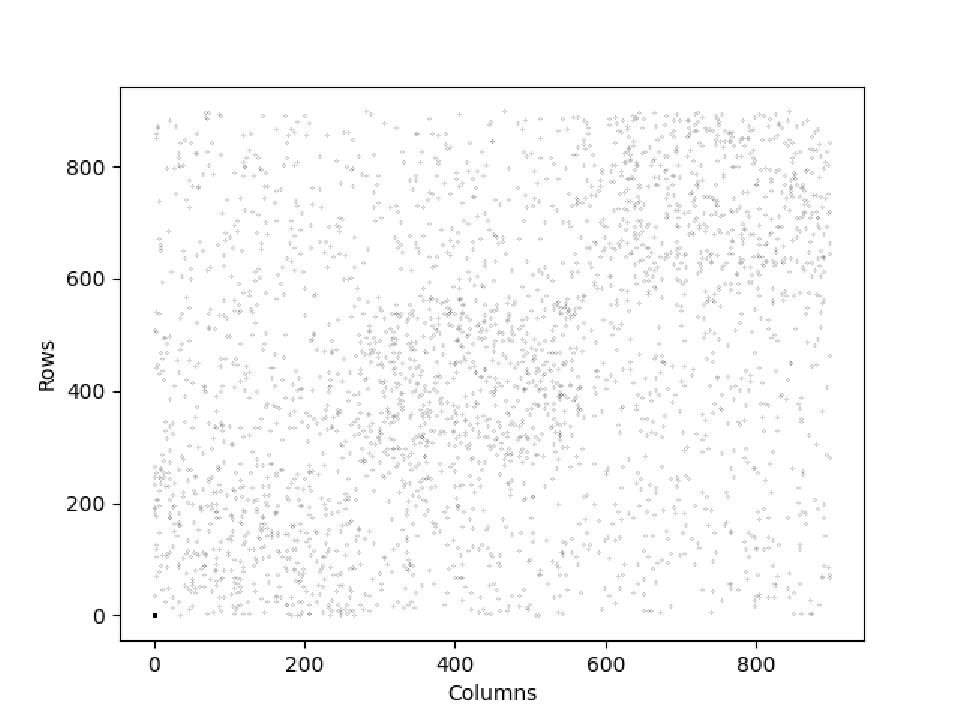} 
\end{minipage}
}
\subfigure[Spectrum of matrix $\K$.]{
\begin{minipage}{5.6cm}
\centering
\includegraphics[scale=0.4]{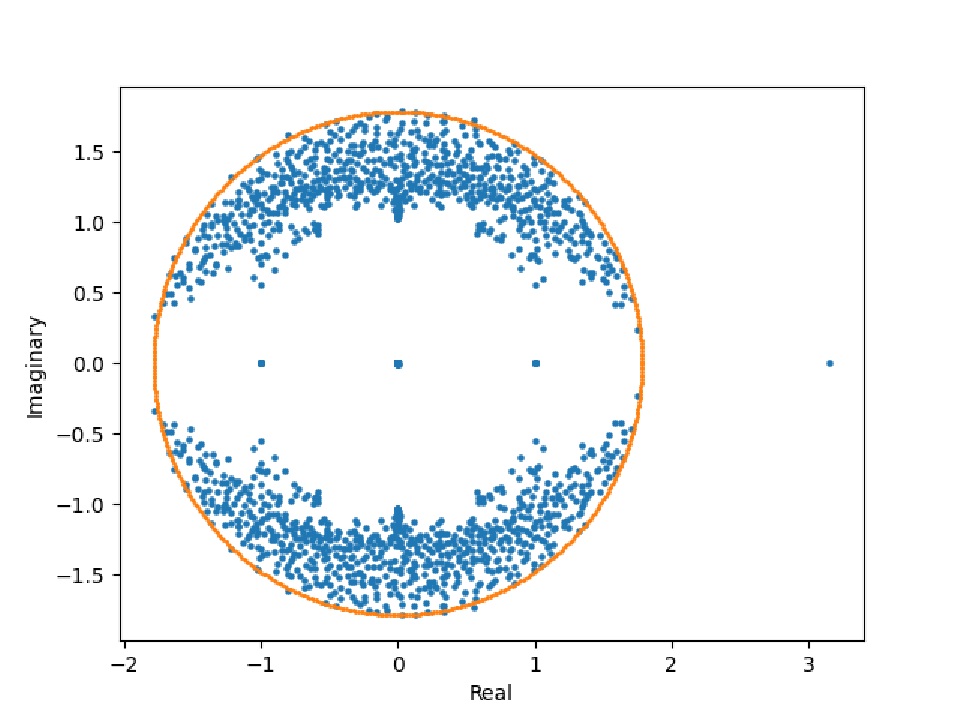}
\end{minipage}
}
\caption{$\beta_1<\beta=0.191<\beta_2$} 
\label{fig:2}   
\end{figure}

\begin{figure}[htbp]
\centering  
\subfigure[Adjacency matrix (black: 1, white: 0).]{   
\begin{minipage}{5.6cm}
\centering    
\includegraphics[scale=0.4]{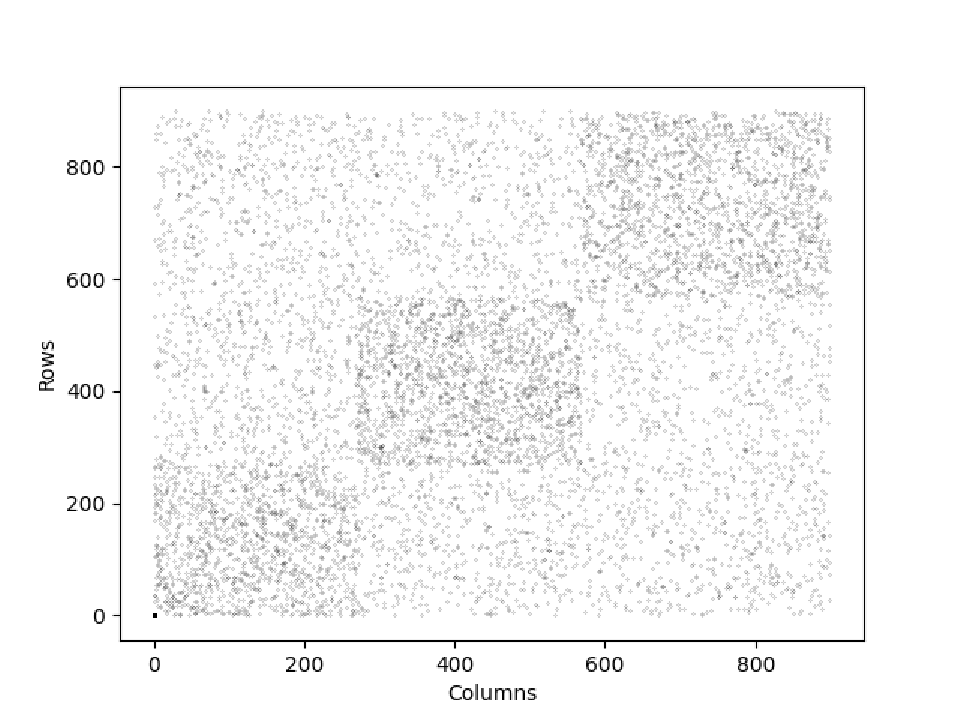} 
\end{minipage}
}
\subfigure[Spectrum of matrix $\K$.]{
\begin{minipage}{5.6cm}
\centering   
\includegraphics[scale=0.4]{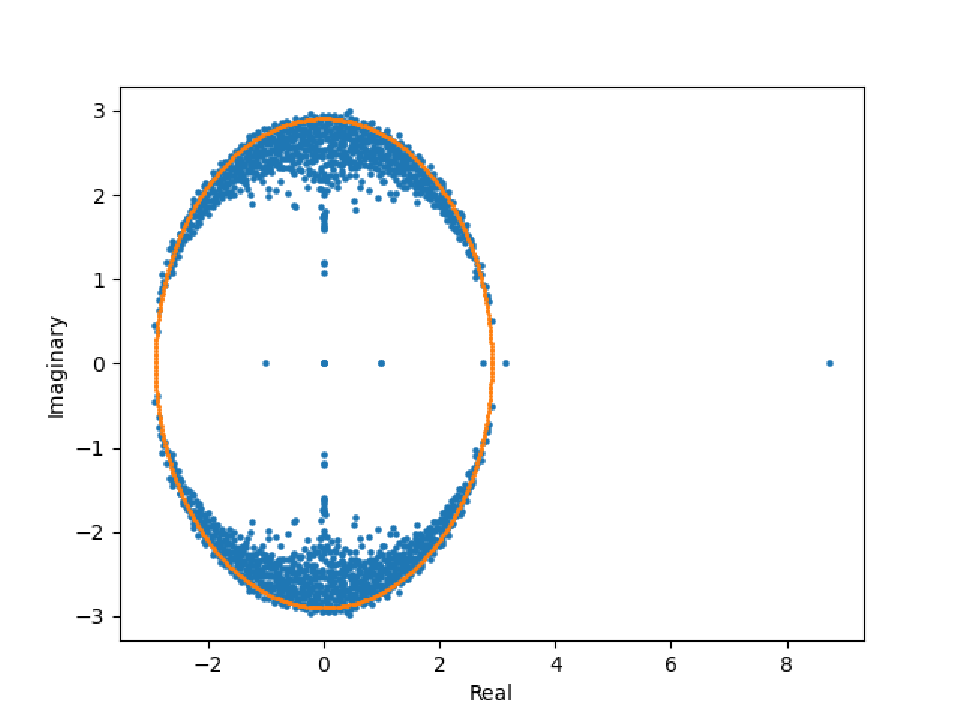}
\end{minipage}
}
\caption{$\beta =\beta_3=0.505$}
\label{fig:5} 
\end{figure}

\begin{figure}[htbp]
\centering  
\subfigure[Adjacency matrix (black: 1, white: 0).]{   
\begin{minipage}{5.6cm}
\centering   
\includegraphics[scale=0.4]{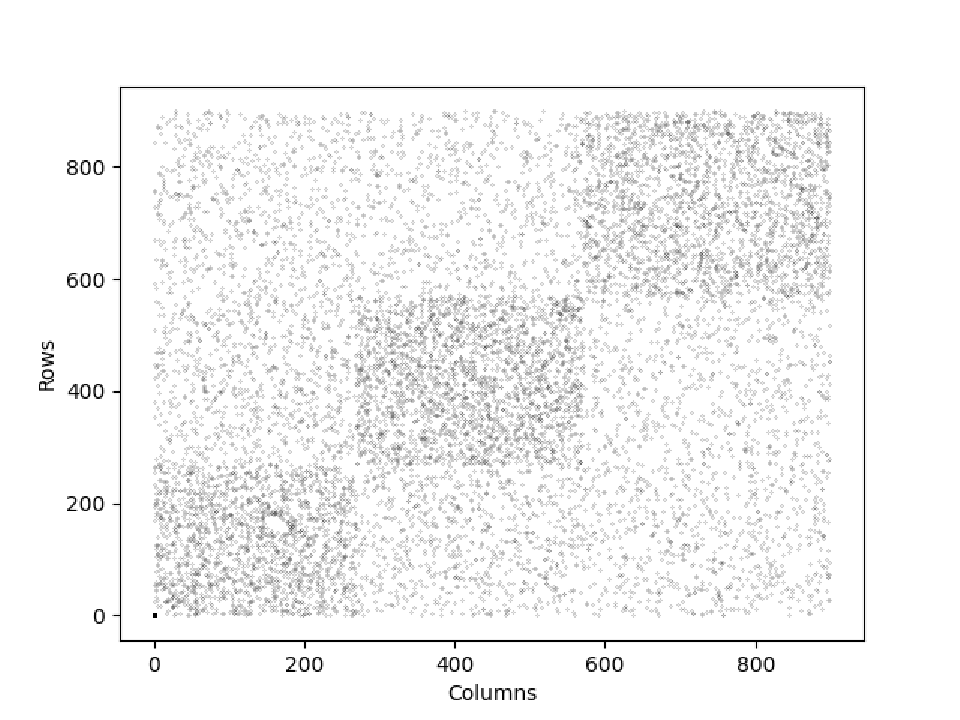} 
\end{minipage}
}
\subfigure[Spectrum of matrix $\K$.]{
\begin{minipage}{5.6cm}

\centering    
\includegraphics[scale=0.4]{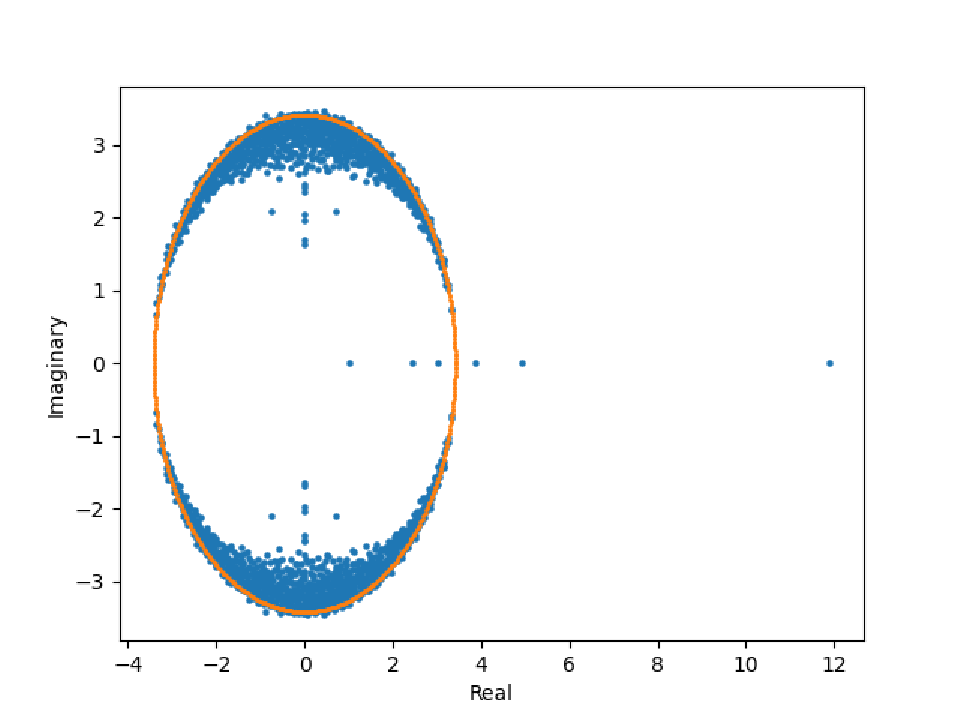}
\end{minipage}
}
\caption{$\beta_3<\beta=0.7$}
\label{fig:7}
\end{figure}

\begin{figure}[htbp]
\centering  
\subfigure[Adjacency matrix (black: 1, white: 0).]{   
\begin{minipage}{5.6cm}
\centering   
\includegraphics[scale=0.4]{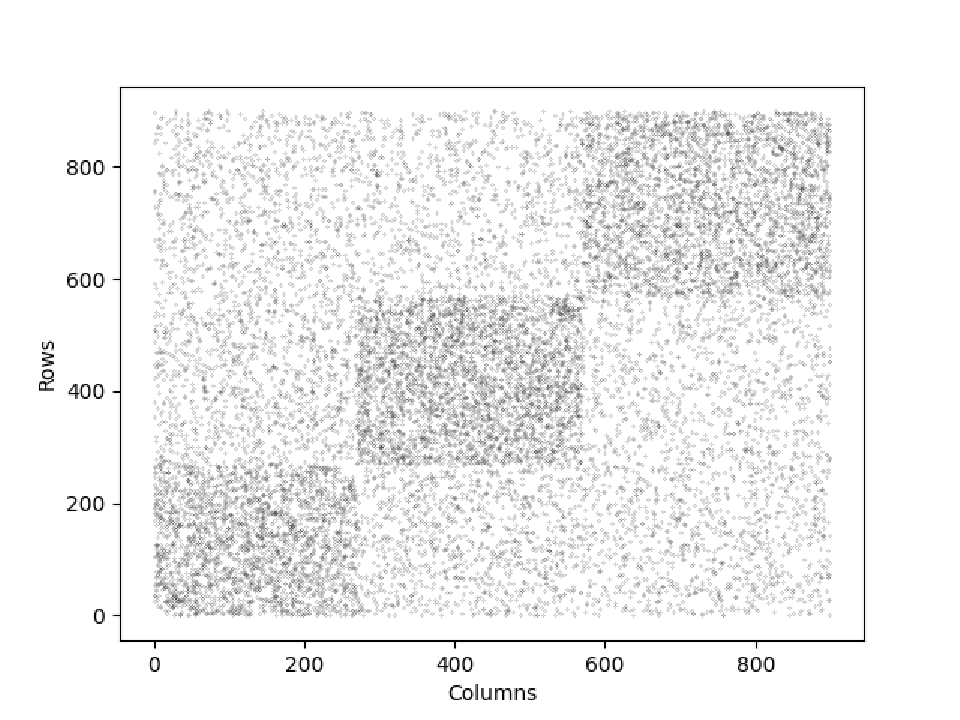} 
\end{minipage}
}
\subfigure[Spectrum of matrix $\K$.]{
\begin{minipage}{5.6cm}

\centering    
\includegraphics[scale=0.4]{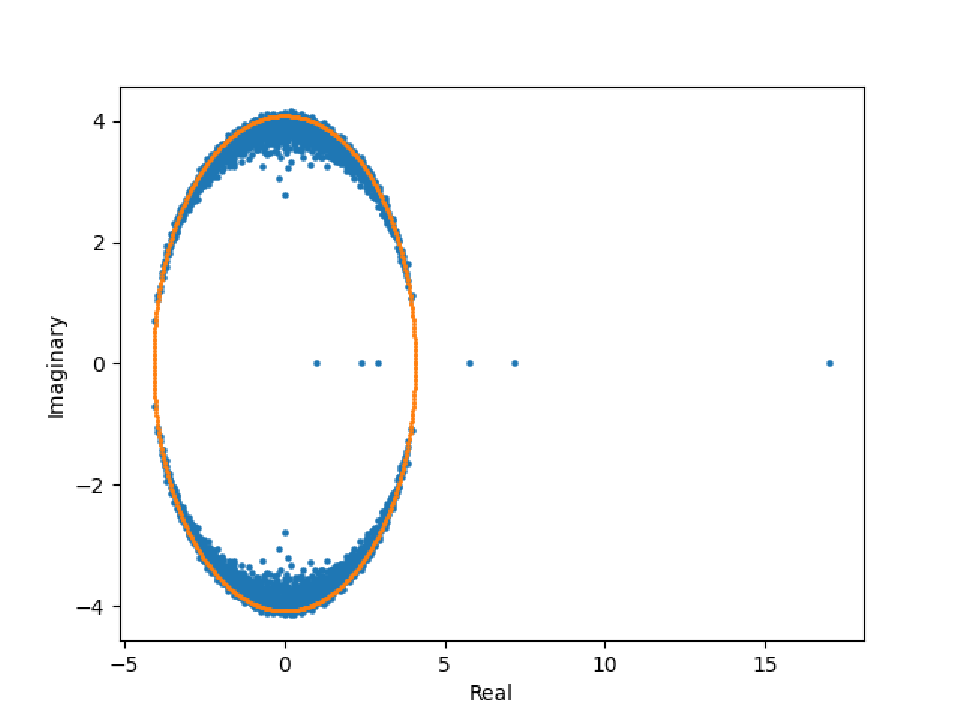}
\end{minipage}
}
\caption{$\beta=1$}   
\label{fig:8}    
\end{figure}

Next a random graph was generated on $n=900$ nodes, with parameter matrices
$\RR =\diag (\frac{35}{107},\frac{42}{107},\frac{30}{107})$ and 
$$
 \C =\begin{pmatrix}
 30  & 11.28 & 7.728 \\
 11.28  & 25 & 10.36 \\
 7.728 & 10.36 & 35
\end{pmatrix} 
$$
constructed so that the average degrees of the clusters be the same, i.e.,
 $c_a =\sum_{b=1}^3 r_b c_{ab}$ is the same for $a=1,2,3$. 

 Figures~\ref{fig:s2},~\ref{fig:s3},~\ref{fig:s4}, and~\ref{fig:s7} show the
 adjacency matrices and the $\K$-spectrum of the $\beta$-percolated
 random graph, when 1,2, and 3 structural (positive real) eigenvalues appear.
Here the phase transitions are closer to $\beta_i$'s %$(i=1,2,3)$ 
than in the previous generic case.

\begin{figure}[htbp]
\centering  
\subfigure[Adjacency matrix (black: 1, white: 0).]{   
\begin{minipage}{5.6cm}
\centering  
\includegraphics[scale=0.4]{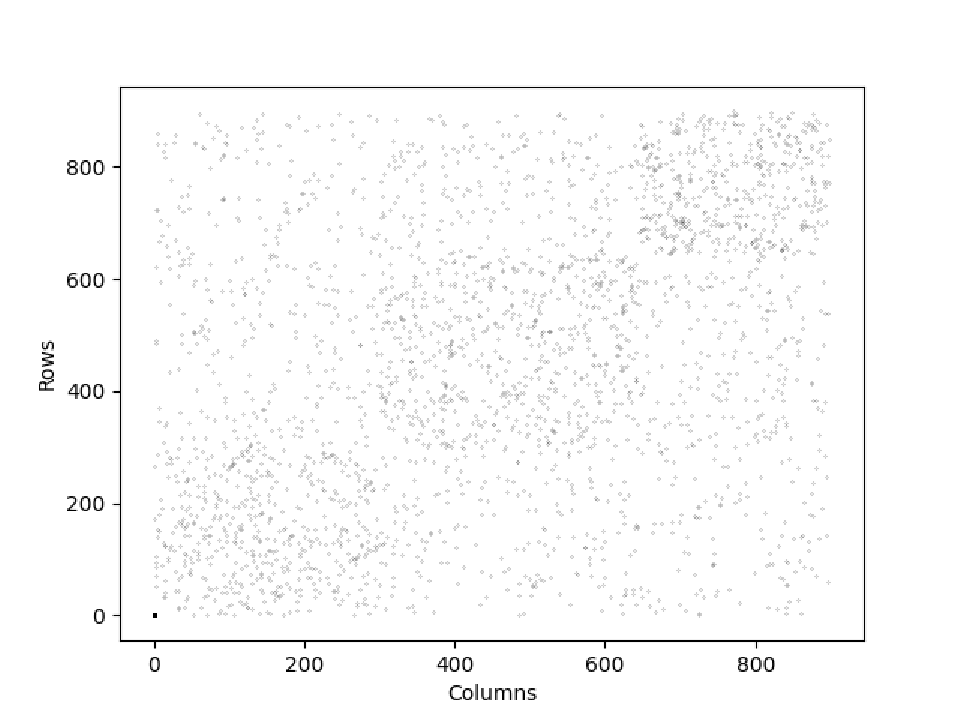} 
\end{minipage}
}
\subfigure[Spectrum of matrix $\K$.]{
\begin{minipage}{5.6cm}
\centering
\includegraphics[scale=0.4]{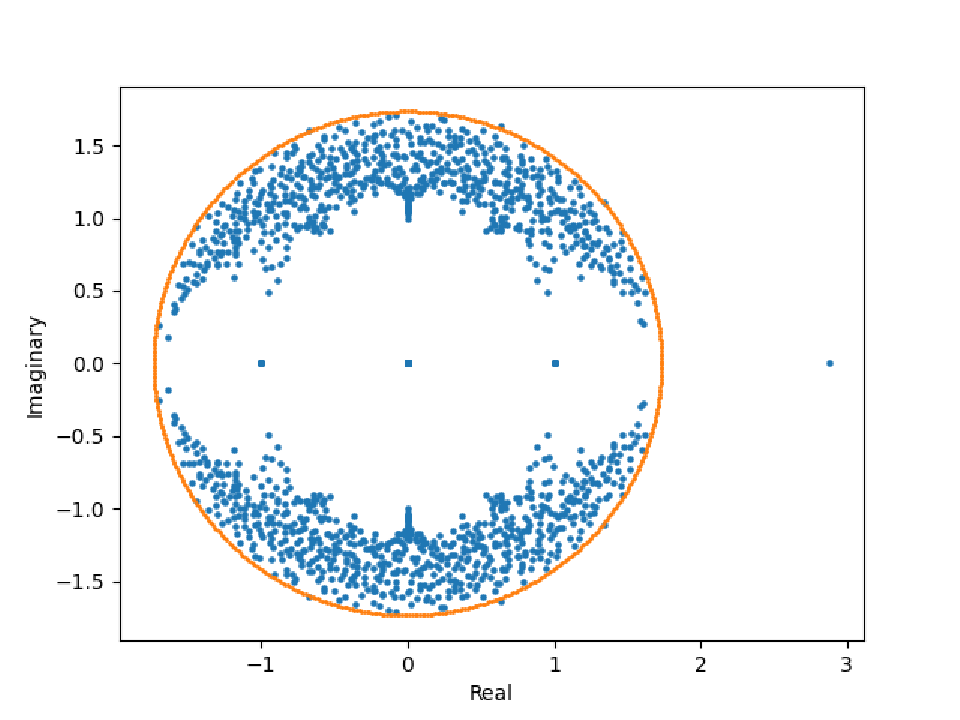}
\end{minipage}
}
\caption{$\beta_1<\beta=0.183<\beta_2$}   
\label{fig:s2}   
\end{figure}

\begin{figure}[htbp]
\centering  
\subfigure[Adjacency matrix (black: 1, white: 0).]{    
\begin{minipage}{5.6cm}
\centering   
\includegraphics[scale=0.4]{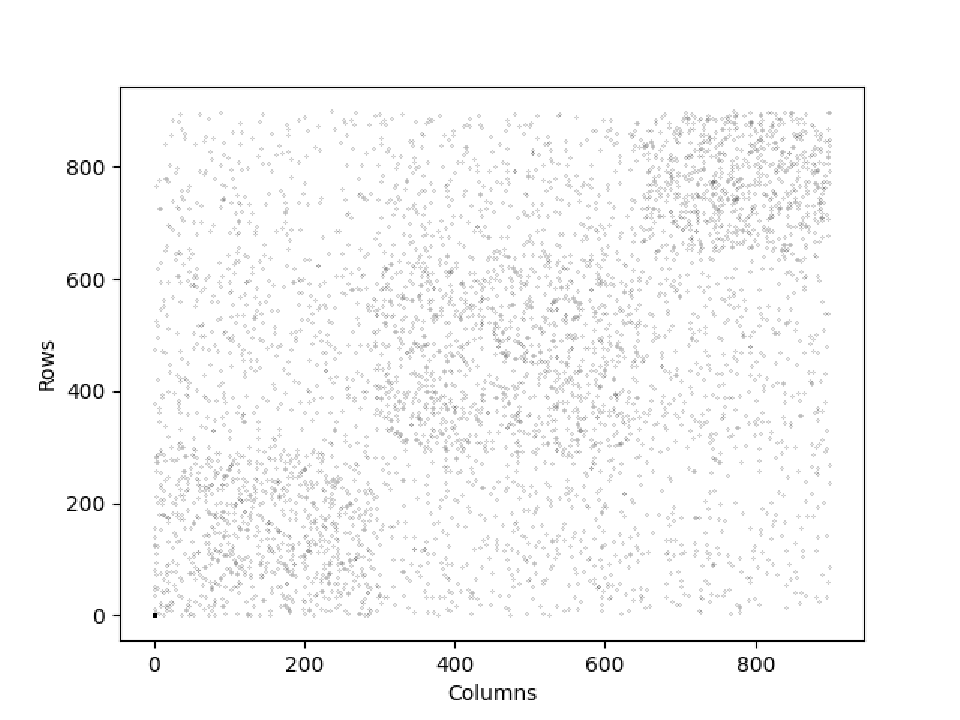}  
\end{minipage}
}
\subfigure[Spectrum of matrix $\K$.]{
\begin{minipage}{5.6cm}
\centering   
\includegraphics[scale=0.4]{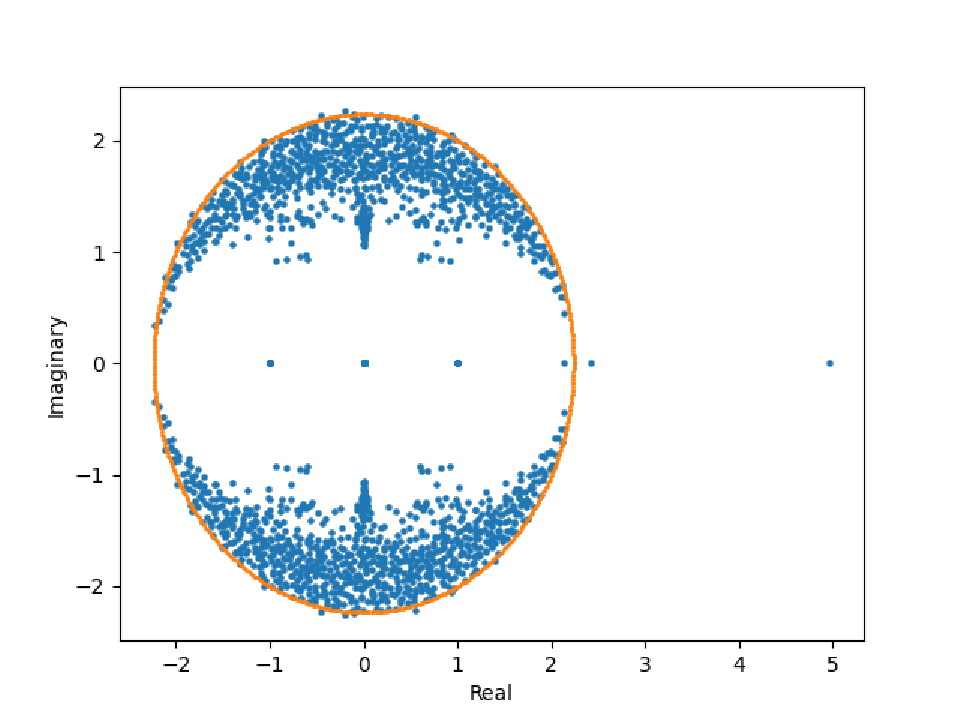}
\end{minipage}
}
\caption{$\beta =\beta_2=0.305$}   
\label{fig:s3}  
\end{figure}

\begin{figure}[htbp]
\centering 
\subfigure[Adjacency matrix (black: 1, white: 0).]{    
\begin{minipage}{5.6cm}
\centering  
\includegraphics[scale=0.4]{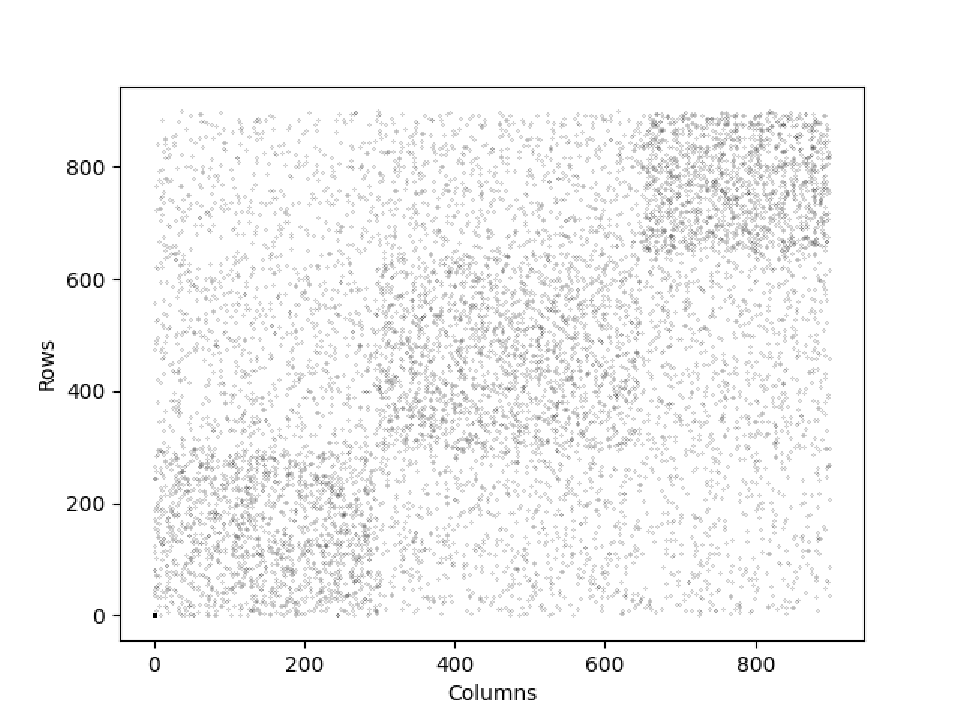}  
\end{minipage}
}
\subfigure[Spectrum of matrix $\K$.]{
\begin{minipage}{5.6cm}
\centering   
\includegraphics[scale=0.4]{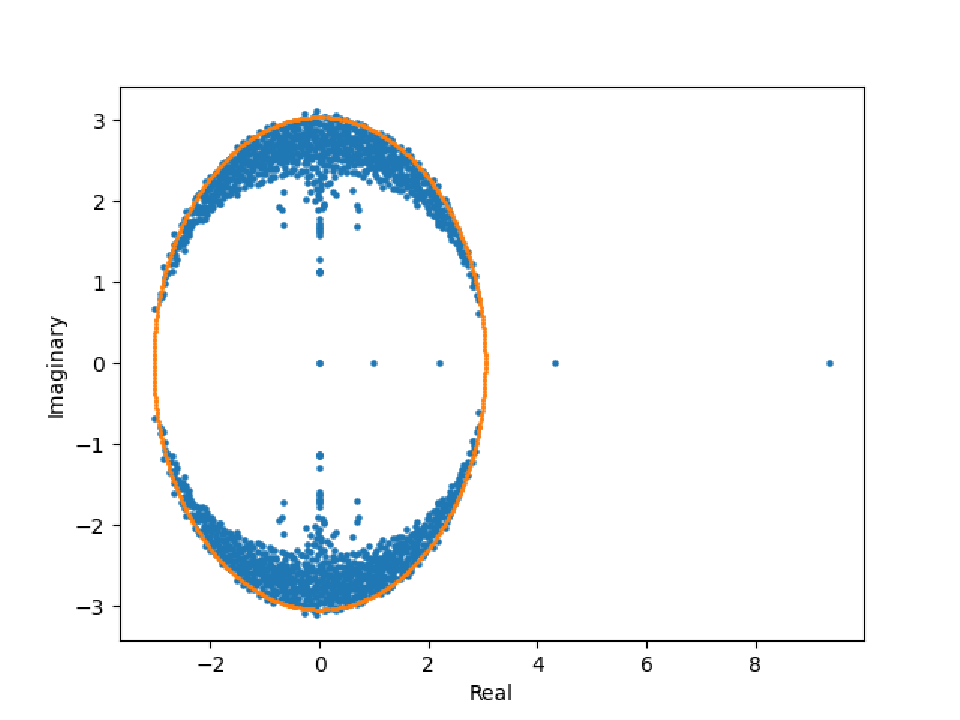}
\end{minipage}
}
\caption{$\beta_2<\beta=0.563<\beta_3$}    
\label{fig:s4}    
\end{figure}

\begin{figure}[htbp]
\centering  
\subfigure[Adjacency matrix (black: 1, white: 0).]{ 
\begin{minipage}{5.6cm}
\centering   
\includegraphics[scale=0.4]{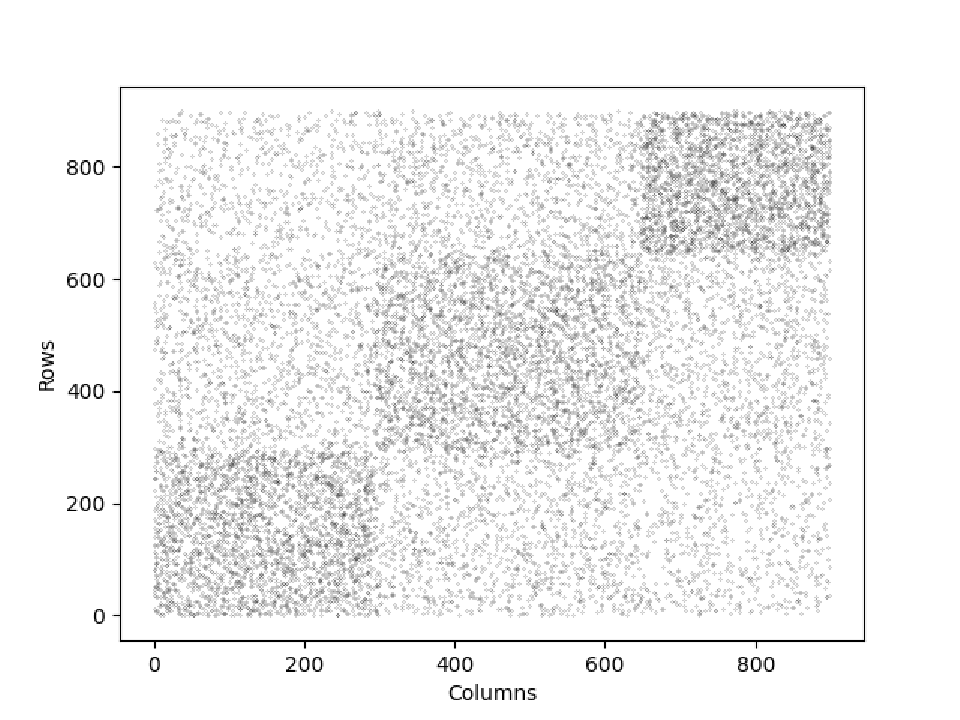} 
\end{minipage}
}
\subfigure[Spectrum of matrix $\K$.]{
\begin{minipage}{5.6cm}

\centering    
\includegraphics[scale=0.4]{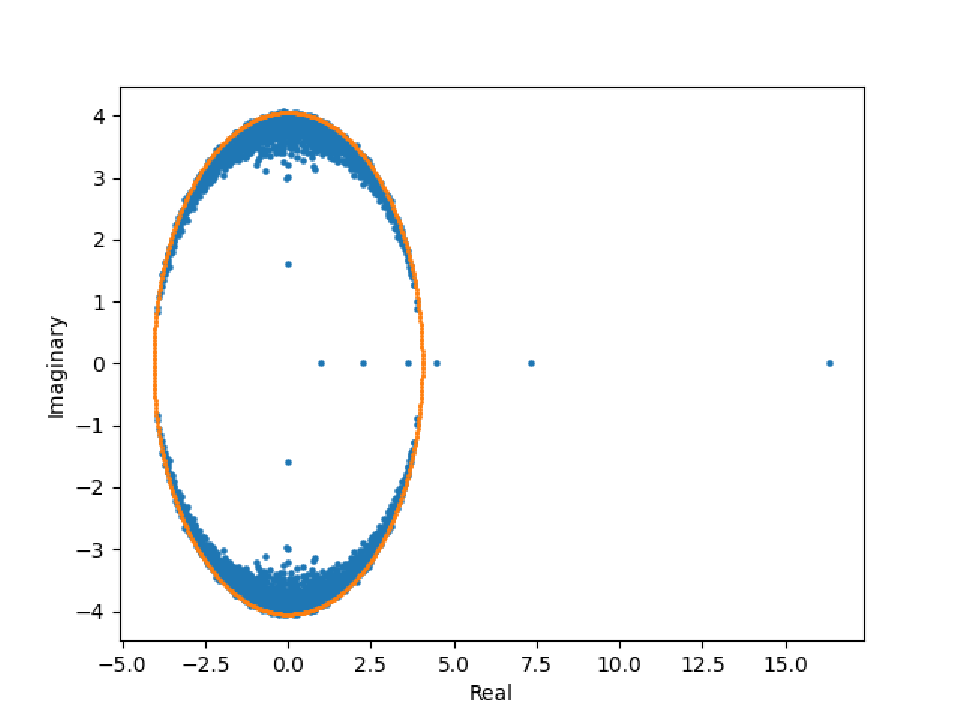}
\end{minipage}
}
\caption{$\beta=1$}   
\label{fig:s7}    
\end{figure}
 
Summarizing, at $\beta_1$, the giant component is formed. 
In~\cite{Keliger}, the effect of different seeding patterns is investigated
in context of the COVID19 pandemic in Hungary.
According to this, below a certain $\beta$, the central seeding 
is more dangerous than the uniform one. Indeed, below $\beta_1$
% the graph is not connected,
there are many small components,  among
which there is a central one. Above that, the
uniform seeding is more dangerous than the central one, as it is able to infect
the giant component and other components too.
However, if the seeds arrive in time, then actions can be done to separate
the most infected clusters as follows.

At $\beta_2$,
there are 2 well-distinguishable clusters, \dots, at 
$\beta_k$, there are $k$ well-distinguishable ones, since the
within-cluster bonds are strengthened. It does not mean that there are
no between-cluster bonds, but those are negligible
compared to the within-cluster ones.
Also, by deleting the between-cluster bonds, the infection can be confined
to the clusters. If the seed is in one cluster, the epidemic can be confined
to that cluster. So with increasing $\beta$, there is a chance to localize
the epidemic to one large city and to save the others.

For example, if $\beta$ passes over $\beta_2$, then 2 clusters can be
separated, and if the seed is in one of them, then the infection can be
localized to that cluster and we can save the other by deleting the
connecting bonds (even if we do not delete all, the epidemic will spread
to the other cluster with small probability only).
If $\beta$ passes over $\beta_3$, then 3 clusters can be
separated, and if the seed is in one of them, then the infection can be
localized to that cluster and we can save the others by deleting the
connecting bonds, etc.
Above $\beta_k$ the situation is similar to that of $\beta =1$, the uniform
seeding is dangerous, but seeds can be localized to the first infected cluster.
In the symmetric case, $\beta_2 =\dots =\beta_k$, so all these phase
transitions occur at the same time from the giant cluster to the $k$-cluster
scenario.
Multiple transitions are spectacular if the eigenvalues of $\B$ greater than
$\sqrt{c}$ are separated from each other and from $\sqrt{c}$.
Therefore, $\beta$ can as well be considered as a tuning parameter. 
Also note that in the $SBM_k$ model when $c_a =c$ $(a=1,\dots ,k )$,
then $\mu_1$ is close to $c$ and so, $\beta_1$ is close to $\frac1{\mu_1}$
(w.h.p.). This is always the case if $k=1$, and so, $\frac{c}{\mu_1^2} =
\frac1{\mu_1}$ as in~\cite{Newman22}.

\section{Application to large quantum chemistry networks}\label{appl}

Computational quantum chemistry aims at computing properties of molecules and materials based on quantum mechanics of many-electron systems,
see, e.g.~\cite{Arendas,Szabo}. In quantum mechanics, states of a system are states in a Hilbert space and relevant physical quantities correspond to some linear operators in such spaces. In a molecule, energy levels of
a many-electron system  are eigenvalues of the Hamiltonian operator, which is quantum version of a classical energy function also known as Hamilton function in classical mechanics. In case of many-electron molecules, the corresponding Hilbert space is an infinite dimensional functional space of so called many-electron wave-functions. For numerical computations, various finite dimensional subspaces may be used (see~\cite{Szabo}). In such subspaces, Hamiltonian operator is represented by the Hamiltonian matrix. The approximate energy levels of a molecule are the real eigenvalues of this matrix, which is real, symmetric.  

When a molecule has more than just few electrons, the corresponding subspace needs to be very large to achieve a good approximation. To a Hamiltonian matrix one can associate a weighted graph, in which absolute values of matrix-elements are weights between nodes and the nodes correspond to basis functions in the subspace. Finding structures in such a network can be useful. In  \cite{Mniszewski} it was suggested that finding network communities allow reduction of the matrix size needed for eigenvalue computations. The Hamiltonian matrix is typically sparse, each row with $n$ elements has approximately $\log^4n$ non-zero elements. As a result, in chemistry we naturally encounter very large and sparse weighted networks. Further, a skeleton graph of the corresponding weighted network is found by rounding non-zero elements to $1$s. 

As a sample we consider the Hamiltonian matrix of the LiH (Lithium Hydride) molecule with $7638$ rows and columns. The corresponding skeleton graph has average degree about $459.559$, which means that on average a row of the adjacency matrix has $94\%$ zero elements. The largest real eigenvalue of the matrix $\K$ is
about $464.473$, which is similar in magnitude to the average degree, an indicator of sparsity of the skeleton graph. The results are shown in Figure~\ref{lih}.
In panel (a), the $18$ largest eigenvalues of the matrix $\K$ are presented,
excluding the trivial top eigenvalue ($\lambda_1\approx 464.473$) that
mirrors the average degree of the skeleton graph. A large gap between the
$8$th and $9$th largest eigenvalues reveals the number of clusters,
which is $8$.
Panel (b) shows the clustering of the vertices in $8$ clusters, by the in-vectors, see Theorem~\ref{kvar}. Rows and columns of each cluster are stacked together.  Intensity of gray-level corresponds to density of $1$s, darker areas having larger density than lighter areas.

\begin{figure}[htbp]
\centering  
\subfigure[$18$ largest real eigenvalues ($\lambda_k$) of the matrix $\K$,
excluding the top eigenvalue ($\lambda_1\approx 464.473$). A large gap between the $8$th and $9$th largest eigenvalues reveal number of clusters,
which is $8$.]{  
\begin{minipage}{5.6cm}
\centering  
\includegraphics[scale=0.65]{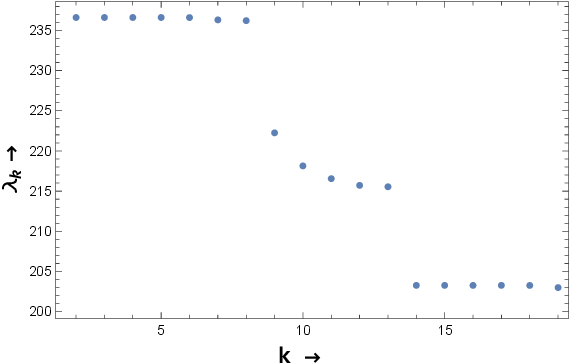} 
\vspace*{1mm}
\end{minipage}
}
\hfill
\subfigure[Clustering the vertices into $8$ clusters by the in-vectors,
illustrated by the corresponding subdivision of the adjacency matrix.  Rows and columns of each cluster are stacked together.  Intensity of gray-level corresponds to density of $1$s, darker areas having larger density than lighter ones.]{
\begin{minipage}{5.6cm}
\centering
\includegraphics[scale=0.49]{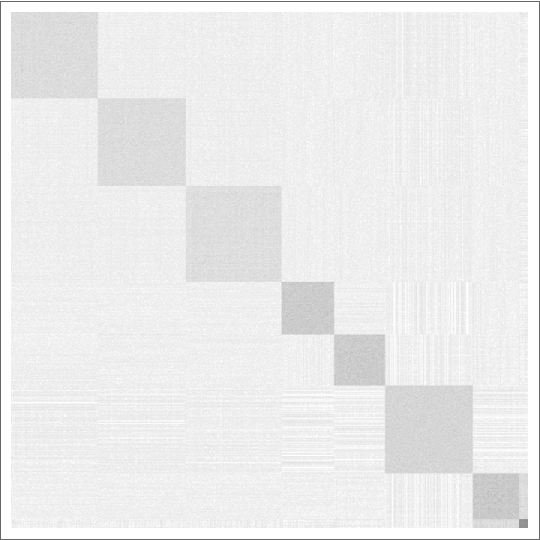}
\vspace*{2mm}
\end{minipage}
}
\caption{Spectral clustering of the skeleton graph of the LiH molecule chemical network with $7638$ nodes.} 
\label{lih}   
\end{figure}

\section{Conclusions and further perspectives}\label{conc}

The sparse stochastic block model 
behind an expanding graph sequence was considered. Here the spectral properties
of the adjacency matrix give less information for the clusters than in the
dense case of~\cite{Bolla20,LovSos}; instead, the non-backtracking spectra
are used.

Bond-percolation is considered in two senses:
the first is that the within- and between-cluster edge probabilities
decrease with $n$, and the second is that edges coming into existence in this 
way are retained with probability $\beta$. 
Similarities and differences between the BP and EM algorithms
are discussed. Both are based on an iteration algorithm, but former one 
solves a system of equations with the transmission probabilities between the
node pairs ($2mk$ non-linear equations and unknowns, where $k$ is the number 
of clusters),  while keeping the model parameters fixed;
latter one, for fixed $k$,
estimates the model parameters and the missing memberships of 
the nodes by using the classical EM iteration for mixtures.
Furthermore, via inflation--deflation techniques and findings of~\cite{Stephan},
we are able to classify the nodes of the sparse graph with representation
based k-means clustering, which gives rise to a sparse spectral clustering
procedure.
Simulation results for the so-called $\beta$-percolation are also considered,
and strategies are suggested for possible localization of a pandemic.

If the $c_{in}$ versus $c_{out}$ scenario does not hold, then in generic
assortative networks there are $k$ different
within-cluster and $k \choose 2$ different between-cluster affinities;
former ones,
denoted by $c_w$'s, are significantly ``larger'' than the latter ones,
denoted by $c_b$'s. In this case,
$c_1 =\dots =c_k =c$ is not
necessarily equivalent to $r_1 =\dots =r_k$, and the
Kesten--Stigum threshold should be developed by giving
lower bounds for every $c_w -c_b$ difference, which are positive
in assortative networks.
Finding such dense diagonal blocks is an important problem of
quantum chemistry, for which purpose, a large network is analyzed.
In the estimation, discrepancy techniques (see, e.g.,~\cite{Bolla16}) may help.

More generally, we plan to investigate non-backtracking spectra of
edge-weighted graphs, that can be considered as a generalized $\beta$-model:
the pandemic is submitted with probability $w_{ij}$ between individual pairs,
instead of the unique $\beta$, that gives rise to investigate an
${SBM}_k^{\W}$ model with edge-weight matrix $\W =(w_{ij})$.

%\vskip0.1cm
%\noindent
%  \textbf{Possible workflow in the unweighted case}

%  \begin{itemize}
%  \item
%    Find $k$ (initial number of clusters) based on the spectral gap in $\B$
 %   (suggested by BP).
 % \item
 %   Run the EM algorithm to estimate the parameters of the $k$-cluster model.
 %   Investigate hypotheses on the model fit (maximum likelihood ratio test,
 %   information theoretical criteria).
 % \item
 %   In case of ``good'' fit, find the clusters with the help of the
  %  vertex representatives (deflated $\B$-eigenvectors), segments of $\K$.
  % \end{itemize} 

%\section*{Author contributions}

%The theoretical parts of the paper with theorems and proofs were written by 
%Marianna Bolla. The Python code
%for calculating the non-backtracking spectra and generating random graphs from
%the $SBM_k$ model was written by Daniel Zhou. He also performed the
%$\beta$-percolation on the so generated random graphs.

\section*{Acknowledgments}

This research was supported by the NKFIH project Dynamical systems and
fractals, no.142169.
The second author acknowledges financial support by Business Finland
BF-COHQCA project.

The research was partly done under the auspices of the Budapest Semesters in 
Mathematics program, in the framework of a research course in the spring
semester 2023, where the third author made simulation experiments.
%The problem itself is based on the research of the DYNASNET group  
%(R\'enyi Institute of Mathematics, Budapest) related to
%bond-percolation and spreading epidemics in social networks. 
The first author is indebted to L\'aszl\'o Lov\'asz for valuable conversations
on this topic, and
also thanks Tam\'as M\'ori, Katalin Friedl and Daniel Keliger for their
useful remarks.

\end{document}